\documentclass[11pt]{amsart}

\usepackage{latexsym}
\usepackage{amssymb}
\usepackage{amsfonts}
\usepackage{graphicx}
\usepackage{epsf}
\usepackage[all]{xypic}

\usepackage{amsthm}

\usepackage{amsmath}
\usepackage{amscd}

\newtheorem{theorem}{Theorem}[section]
\newtheorem{corollary}[theorem]{Corollary}
\newtheorem{lemma}[theorem]{Lemma}
\newtheorem{proposition}[theorem]{Proposition}

\newtheorem{example}[theorem]{Example}
\newtheorem{remark}[theorem]{Remark}

\def\bea{\begin{eqnarray*}}
\def\eea{\end{eqnarray*}}

\def\ot{\otimes}
\def\ra{\rightarrow}

\def\al{\alpha}

\def\bea{\begin{eqnarray*}}
\def\eea{\end{eqnarray*}}

\def\rhu{\rightharpoonup}
\def\lhu{\leftharpoonup}

\begin{document}

\title[Picard groups and quasi-Frobenius algebras]{Picard groups of quasi-Frobenius algebras and a question
on Frobenius strongly graded algebras}
\author[S. D\u{a}sc\u{a}lescu, C. N\u{a}st\u{a}sescu and L. N\u{a}st\u{a}sescu]{ Sorin D\u{a}sc\u{a}lescu$^1$, Constantin
N\u{a}st\u{a}sescu$^2$ and Laura N\u{a}st\u{a}sescu$^{2,3}$}
\address{$^1$ University of Bucharest, Faculty of Mathematics and Computer Science,
Str. Academiei 14, Bucharest 1, RO-010014, Romania} \address{ $^2$
Institute of Mathematics of the Romanian Academy, PO-Box
1-764\\
RO-014700, Bucharest, Romania}
\address{ $^3$
Max Planck Institut f${\rm \ddot{u}}$r Mathematik, Vivatsgasee 7,
53111 Bonn, Germany}
\address{
 e-mail: sdascal@fmi.unibuc.ro, Constantin\_nastasescu@yahoo.com,
 lauranastasescu@gmail.com
}

\date{}
\maketitle

\begin{abstract}
Our initial aim was to answer the question: does the Frobenius
(symmetric) property transfer from a strongly graded algebra to its
homogeneous component of trivial degree?  Related to it, we
investigate invertible bimodules and the Picard group of a finite
dimensional quasi-Frobenius algebra $R$. We compute the Picard
group, the automorphism group and the group of outer automorphisms
of a $9$-dimensional quasi-Frobenius algebra which is not Frobenius,
constructed by Nakayama. Using these results and a semitrivial
extension construction, we give an example of a symmetric strongly
graded algebra whose trivial homogeneous component is not even
Frobenius.  We investigate associativity of isomorphisms
$R^*\ot_RR^*\simeq R$ for quasi-Frobenius algebras $R$, and we
determine the order of the class of the invertible bimodule $H^*$ in
the Picard group of a finite-dimensional Hopf algebra $H$.
As an application, we construct new examples of symmetric algebras.\\
2020 MSC: 16D50, 16D20, 16L60, 16S99, 16T05, 16W50\\
Key words: quasi-Frobenius algebra, Frobenius algebra, symmetric
algebra, invertible bimodule, Picard group, strongly graded algebra,
Hopf algebra, Nakayama automorphism.
\end{abstract}

\section{Introduction and preliminaries}

A finite-dimensional algebra $A$ over a field $K$ is called
Frobenius if $A\simeq A^*$ as left (or equivalently, as right)
$A$-modules. If $A$ satisfies the stronger condition that $A\simeq
A^*$ as $A$-bimodules, then $A$ is called a symmetric algebra.
Frobenius algebras and symmetric algebras occur in algebra,
geometry, topology and quantum theory, and they have a rich
representation theory, which is relevant for all these branches of
mathematics.  A general problem is whether a certain ring property
transfers from an algebra on which a Hopf algebra (co)acts to the
subalgebra of (co)invariants; of special interest is the situation
where the (co)action produces a Galois extension. Particular cases
of high relevance are: (1) algebras $A$ on which a group $G$ acts as
automorphisms, and the transfer of properties to the subalgebra
$A^G$ of invariants; (2) algebras $A$ graded by a group $G$, and the
transfer of properties to the homogeneous component of trivial
degree. In the second case, such an $A$ is in fact a comodule
algebra over the Hopf group algebra $KG$, and the subalgebra of
coinvariants is just the component of trivial degree; moreover, the
associated extension is $KG$-Hopf-Galois if and only if $A$ is
strongly graded.

Our main aim is to answer the
following.\\

{\bf Question 1.} {\it If $A=\oplus_{g\in G}A_g$ is a strongly
$G$-graded algebra, where $G$ is a group with neutral element $e$,
and $A$ is Frobenius (symmetric), does it follow that the subalgebra
$A_e$ is Frobenius (symmetric)?}\\

In Section \ref{sectionassociative} we answer the question in the
negative, for both Frobenius and symmetric properties. There is an
interesting alternative way to formulate this question for the
Frobenius property. Frobenius algebras in the monoidal category of
$G$-graded vector spaces were considered in \cite{dnn1}, where they
were called graded Frobenius algebras. Such objects and a shift
version of them occur in noncommutative geometry, for example as
Koszul duals of certain Artin-Schelter regular algebras, and also in
the theory of Calabi-Yau algebras. A $G$-graded algebra $A$ is
graded Frobenius if $A\simeq A^*$ as graded left $A$-modules, where
$A^*$ is provided with a standard structure of such an object.
Obviously, if $A$ is graded Frobenius, then it is a Frobenius
algebra, while the converse is not true in general. If $A$ is
strongly graded, then $A$ is graded Frobenius if and only if $A_e$
is Frobenius, see \cite[Corollary 4.2]{dnn1}. Thus the question
above can be also formulated as: If $A$ is a strongly graded algebra
which is Frobenius, is it necessarily graded Frobenius?

Question 1 cannot be reformulated in a similar way for the symmetric
property. As $KG$ is a cosovereign Hopf algebra with respect to its
counit, see \cite{bichon} for details, a concept of symmetric
algebra can be defined in its category of corepresentations, i.e.,
in the monoidal category of $G$-graded vector spaces; the resulting
objects are called graded symmetric algebras. As expected, $A$ is
graded symmetric if $A\simeq A^*$ as graded $A$-bimodules. If $A$ is
strongly graded, then $A_e$ is symmetric whenever $A$ is graded
symmetric, however the converse is not true, see \cite[Remark
5.3]{dnn1}. This shows that Question 1 is not equivalent to asking
whether a symmetric strongly graded algebra is graded symmetric,
nevertheless this other question is also of interest.

The transfer of the Frobenius property from the strongly graded
algebra $A$ to $A_e$ works well under additional conditions, for
example if $A$ is free as a left and as a right $A_e$-module, in
particular if $A$ is a crossed product of $A_e$ by $G$, see
\cite{dnn2}. If $A$ is Frobenius, then it is left (and right)
self-injective, and then so is $A_e$; this means that $A_e$ is a
quasi-Frobenius algebra. Thus a possible example answering Question
1 in the negative should be built on a quasi-Frobenius algebra which
is not Frobenius. Moreover, by Dade's Theorem, each homogeneous
component of the strongly graded algebra $A$ is an invertible
$A_e$-bimodule, see \cite{nvo}, suggesting a study of the Picard
group ${\rm Pic}(A_e)$ of $A_e$. In Section \ref{sectionPicardQF} we
look at invertible bimodules over a finite-dimensional
quasi-Frobenius algebra $R$. For such an $R$, an object of central
interest is the linear dual $R^*$ of the regular bimodule $R$; we
show that it is an invertible $R$-bimodule. In the case where $R$ is
Frobenius, $R^*$ is isomorphic to a deformation of the regular
bimodule $R$, with the right action modified by the Nakayama
automorphism $\nu$ of $R$ with respect to a Frobenius form. It
follows that the order of the class $[R^*]$ of $R^*$ in ${\rm
Pic}(R)$ is just the order of the class of $\nu$ in the group ${\rm
Out}(R)$ of outer automorphisms of $R$. If $R$ is not Frobenius,
then $R^*$ cannot be obtained from $R$ by deforming the right action
by an automorphism, or in other words, $[R^*]$ does not lie in the
image of ${\rm Out}(R)$ inside ${\rm Pic}(R)$, and we show that it
lies in the centralizer of ${\rm Out}(R)$. We compute the order of
$[R^*]$ in ${\rm Pic}(R)$ for: (1) liftings of certain Hopf algebras
in the braided category of Yetter-Drinfeld modules, called quantum
lines, over the group Hopf algebra of a finite abelian group; (2)
certain quotients of quantum planes. This order may be any positive
integer, as well it can be infinite. It is known that a
finite-dimensional
Hopf algebra is Frobenius. In this case we  prove the following.\\

{\bf Theorem A.} {\it Let $H$ be a finite-dimensional Hopf algebra
with antipode $S$. Then the order of $[H^*]$ in ${\rm Pic}(H)$ is
the least common multiple of the order of the class of $S^2$ in
${\rm Out}(H)$ and the order of the modular element of $H^*$ in the
group of grouplike elements of $H^*$.}\\

 As a particular case, one gets
a well-known characterization of symmetric finite-dimensional Hopf
algebras, as those unimodular Hopf algebras such that $S^2$ is
inner.

In Section \ref{sectionpicardquasi} we explain that if $R$ is a
finite-dimensional quasi-Frobenius algebra, and $S$ a basic algebra
of $R$, which is necessarily Frobenius, then ${\rm Pic}(R)\simeq
{\rm Pic}(S)$, and moreover, the order of $[R^*]$ in ${\rm Pic}(R)$
is equal to the order of $[S^*]$ in ${\rm Pic}(S)$.

In Section \ref{sectionconstruction} we consider an algebra of
dimension $9$ which is quasi-Frobenius, but not Frobenius, and we
investigate its structure and determine its Picard group. This
algebra was introduced by Nakayama in \cite{nak} in a matrix
presentation, see also \cite[Example 16.19.(5)]{lam}. We use a
different presentation given in \cite{dnn3}. Let $\mathcal{R}$ be
the $K$-algebra with basis ${\bf B}=\{ E, X_1,X_2,Y_1,Y_2\}\cup
\{F_{ij}|1\leq i,j\leq 2\}$, and relations

 \bea E^2=E,&
F_{ij}F_{jr}=F_{ir}\\
EX_{i}=X_{i}, &X_{i}F_{ir}=X_{r}\\
F_{ij}Y_{j}=Y_{i},& Y_{i}E=Y_{i} \eea for any $1\leq i,j,r\leq 2$,
 and any other product of two elements of $\bf B$ is zero. We show
 that any invertible $\mathcal{R}$-bimodule is either a deformation
 of $\mathcal{R}$ or one of $\mathcal{R}^*$ by an automorphism of
 $\mathcal{R}$, and we have an exact sequence
 $$1\ra {\rm Inn}(\mathcal{R})\ra {\rm Aut}(\mathcal{R})\ra {\rm
 Pic}(\mathcal{R})\ra C_2\ra 1,$$
 where $C_2$ is the cyclic group of order $2$. If $V$ is an $\mathcal{R}$-bimodule, and $\alpha$ is an automorphism of $\mathcal{R}$, we denote by
 $_1V_\alpha$ the bimodule obtained from $V$ by changing the right action via $\alpha$. We collect the conclusions of this section in:\\

 {\bf Theorem B.} {\it  $[\mathcal{R}^*]$
has order $2$ in ${\rm Pic}(\mathcal{R})$. An invertible
$\mathcal{R}$-bimodule is isomorphic either
 to
$_1\mathcal{R}_{\alpha}$ or to $_1{\mathcal{R}^*}_{\alpha}$ for some
$\alpha\in {\rm Aut}(\mathcal{R})$, and ${\rm
Pic}(\mathcal{R})\simeq {\rm Out}(\mathcal{R})\times C_2$.}\\

In Section \ref{sectionautomorfisme} we compute the automorphism
group ${\rm Aut}(\mathcal{R})$ and the group ${\rm
Out}(\mathcal{R})$ of outer automorphisms. For this aim, we use
another presentation of $\mathcal{R}$, given in \cite{dnn3}. Thus
$\mathcal{R}$ is isomorphic to the Morita ring associated with a
Morita context connecting the rings $K$ and $M_2(K)$, where the
connecting bimodules are $K^2$ and $M_{2,1}(K)=\left[ \begin{array}{c} K\\
 K \end{array} \right]$ with actions given by the usual matrix
multiplication, and such that both Morita maps are zero. Thus
$\mathcal{R}$ is isomorphic as a linear space to the matrix algebra
$M_3(K)$, but its multiplication is altered by
collapsing the product of the off diagonal blocks. We prove: \\

{\bf Theorem C. }{\it ${\rm Aut}(\mathcal{R})$ is isomorphic to a
semidirect product $(K^2\times M_{2,1}(K))\rtimes (K^{\times}\times
GL_2(K))$, and ${\rm Out}(\mathcal{R})\simeq K^{\times}$.}\\

Here $K^{\times}$ denotes the multiplicative group associated with
$K$. We explicitly describe the automorphisms and the outer
automorphisms. Comparing to the matrix algebra $M_3(K)$, where there
are no outer automorphisms, the alteration of the multiplication
produces non-trivial outer automorphisms of $\mathcal{R}$.  As a
consequence of Theorems B and C, wee see that ${\rm
Pic}(\mathcal{R})\simeq K^{\times}\times C_2$.

In Section \ref{semitrivialextensions} we consider an arbitrary
finite-dimensional algebra $R$ and a morphism of $R$-bimodules
$\psi:R^*\ot_RR^*\ra R$ which is associative, i.e.,
$\psi(r^*\ot_Rs^*)\lhu t^*=r^*\rhu \psi(s^*\ot_Rt^*)$ for any
$r^*,s^*,t^*\in R^*$; here $\rhu$ and $\lhu$ denote the usual left
and right actions of $R$ on $R^*$. Then we can form the semitrivial
extension $R\rtimes_\psi R^*$, which is the cartesian product
$R\times R^*$ with the usual addition, and multiplication defined by
$$(r,r^*)(s,s^*)=(rs+\psi(r^*\ot_Rs^*),(r\rhu s^*)+(r^*\lhu s))$$ for any $r,s\in
R, r^*,s^*\in R^*$. It has a structure of a $C_2$-graded algebra
with $R$ as the homogeneous component of trivial degree. We prove:\\

{\bf Proposition A. }{\it $R\rtimes_\psi R^*$ is a symmetric
algebra.}\\

 If $\psi=0$, we get a well-known construction of
Tachikawa, see \cite{lam}; in this case $R$ may be any finite
dimensional algebra. If $\psi$ is an isomorphism, which implies that
$R^*$ is invertible, then $R\rtimes_\psi R^*$ is a strongly
$C_2$-graded algebra. This suggests that in order to construct
symmetric strongly graded
algebras, it is natural to ask the following.\\

{\bf Question 2. }{\it If $R$ is a finite-dimensional algebra such
that $R^*\ot_RR^*\simeq R$ as $R$-bimodules, is it true that any
isomorphism $\psi:R^*\ot_RR^*\ra R$ is associative?}\\

We address it in Section \ref{sectionassociative}. We will see that
if  $R^*\ot_RR^*\simeq R$, then $R$ is necessarily quasi-Frobenius.
We first answer the question if $R$ is Frobenius, and then we derive
the quasi-Frobenius case by using Morita theory to reduce to the
basic
algebra of $R$. We prove the following.\\

{\bf Proposition B. }{\it Let $R$ be a finite-dimensional algebra
such that $[R^*]$ has order at most $2$ in ${\rm Pic}(R)$. Then any
isomorphism
$\psi:R^*\ot_RR^*\ra R$ is associative.}\\

In particular,  any isomorphism $\varphi:\mathcal{R}^*\ot
_\mathcal{R}\mathcal{R}^*\ra \mathcal{R}$ resulting from Theorem B
 is associative, so then the strongly $C_2$-graded algebra
$\mathcal{R}\rtimes_{\varphi} \mathcal{R}^*$ is symmetric, thus also
Frobenius, while its component of trivial degree is not Frobenius.
This answers in the negative Question 1, for both Frobenius and
symmetric properties. It also answers the other question related to
the symmetric property, since $\mathcal{R}\rtimes_{\varphi}
\mathcal{R}^*$ is symmetric, but it is not graded symmetric as its
homogeneous component of degree $e$ is not symmetric.

 Besides
producing the large class of symmetric algebras presented above, the
semitrivial extension construction is of interest by itself, at
least taking into account the wealth of results of interest
concerning trivial extensions, i.e., those associated to zero
morphisms $\psi$.

More examples answering in the negative Question 1 for the
 symmetric property are obtained for Frobenius algebras $R$ such that $[R^*]$ has order $2$ in
${\rm Pic}(R)$. We present several classes of algebras $R$
 enjoying these properties. Among them, we note that for any finite-dimensional
 unimodular Hopf algebra $H$, $[H^*]$ has order at most
 $2$ in ${\rm Pic}(H)$.

We work over a field $K$, with multiplicative group $K^{\times}$. We
refer to \cite{lam}, \cite{lorenz} and \cite{sy} for facts related
to (quasi-)Frobenius algebras and symmetric algebras,  to \cite{nvo}
for results about graded rings, and to \cite{radford2} for basic
notions about Hopf algebras. We recall that if $G$ is a group with
neutral element $e$, an algebra $A$ is $G$-graded if it has a
decomposition $A=\oplus_{g\in G}A_g$ as a direct sum of linear
subspaces such that $A_gA_h\subset A_{gh}$ for any $g,h\in G$; in
particular, $A_e$ is a subalgebra of $A$. Such an $A$ is called
strongly graded if $A_gA_h=A_{gh}$ for any $g,h\in G$.

\section{Quasi-Frobenius algebras and invertible bimodules}
\label{sectionPicardQF}

We recall from \cite{bass} some basic facts concerning invertible
bimodules and the Picard group. Let $R$ be an algebra over a field
$K$. An $R$-bimodule $P$ is called invertible if it satisfies one of
the following equivalent conditions: (1) There exists a bimodule $Q$
such that $P\ot_RQ$ and $Q\ot_RP$ are isomorphic to $R$ as
bimodules; (2) The functor $P\ot_R -:R-mod\ra R-mod$ is an
equivalence of categories; (3) $P$ is a finitely generated
projective generator as a left $R$-module, and the map $\omega:R\ra
{\rm End}(_RP)$, $\omega (r)(p)=pr$ for any $r\in R, p\in P$ is a
ring isomorphism.

We keep the usual convention that the multiplication in ${\rm
End}(_RP)$ is the inverse map composition. The set of isomorphism
types of invertible $R$-bimodules is a group with multiplication
defined by $[U]\cdot [V]=[U\ot_RV]$, where $[U]$ denotes the class
of the bimodule $U$ with respect to the isomorphism equivalence
relation. This group is called the Picard group of $R$, and it is
denoted by ${\rm Pic}(R)$.

If $V$ is an $R$-bimodule and $\alpha,\beta$ are elements in the
group ${\rm Aut}(R)$ of algebra automorphisms  of $R$, we denote by
$_\alpha V_\beta$ the bimodule with the same underlying space as
$V$, and left and right actions defined by $r\ast v=\alpha(r)v$ and
$v\ast r=v\beta(r)$ for any $v\in V$ and $r\in R$. The following
facts hold for any $\alpha,\beta,\gamma\in {\rm Aut}(R)$. All
isomorphisms are of $R$-bimodules, and $1$ denotes the identity morphism.\\

$\bullet$ $_{\gamma\alpha} R_{\gamma\beta}\simeq {_{\alpha}
R_{\beta}}$, in particular $_{\alpha} R_{\beta}\simeq
{_1R_{\alpha^{-1}\beta}}$.\\

$\bullet$ $_1R_\alpha \ot _{R} {_1R_\beta}\simeq
{_1R_{\alpha\beta}}$, thus $_1R_\alpha$ is invertible, and
$[_1R_\alpha]^{-1}=[_1R_{\alpha^{-1}}]$.\\

$\bullet$ $_1R_\alpha\simeq {_1R_\beta}$ if and only if
$\alpha\beta^{-1}$ is an inner automorphism of $R$, i.e., there
exists an invertible element $u\in R$ such that
$\alpha\beta^{-1}(r)=u^{-1}ru$ for any $r\in R$. Denote by ${\rm
Inn}(R)$ the group of inner automorphisms of $R$. In particular,
$_1R_\alpha \simeq R$ if and only if $\alpha\in {\rm Inn}(R)$, thus
there is an exact sequence of groups $0\ra {\rm
Inn}(R)\hookrightarrow {\rm Aut}(R)\ra {\rm Pic}(R)$, the last
morphism in the sequence taking $\alpha$ to $_1R_\alpha$. The factor
group ${\rm Aut}(R)/{\rm Inn}(R)$, denoted by ${\rm Out}(R)$, is
called the group of outer automorphisms of $R$, and it embeds into
${\rm Pic}(R)$.\\

$\bullet$ $_\alpha V_\beta\simeq {_\alpha R_1}\ot _R V\ot_R
{_1R_\beta}$.\\

We will also need the following.

\begin{proposition} \label{isoinvertible}
{\rm (\cite[page 73]{bass})} Let $U$ and $V$ be invertible
$R$-bimodules such that $U\simeq V$ as left $R$-modules. Then there
exists $\alpha\in {\rm Aut}(R)$ such that $U\simeq {_1V_\alpha}$ as
$R$-bimodules.
\end{proposition}

Now let $V$ be a bimodule over the $K$-algebra $R$. Then the linear
dual $V^*={\rm Hom}_K(V,K)$ is an $R$-bimodule with actions denoted
by $\rhu$ and $\lhu$, given by $(r\rhu v^*)(v)=v^*(vr)$ and
$(v^*\lhu r)(v)=v^*(rv)$ for any $r\in R, v^*\in V^*, v\in V$. One
can easily check that $(_\alpha V_\beta)^*= {_\beta (V^*)_\alpha}$
for any $\alpha,\beta \in {\rm Aut}(R)$. If $V$ is finite
dimensional, then $(V^*)^*\simeq V$, and this shows that two
finite-dimensional bimodules $V$ and $W$ are isomorphic if and only
if so are their duals $V^*$ and $W^*$.

We are interested in a particular bimodule, namely $R^*$, the dual
of $R$. Some immediate consequences of the discussion above are that
for any $\alpha,\beta,\gamma\in {\rm Aut}(R)$:\\

$\bullet$ $_{\gamma\alpha} (R^*)_{\gamma\beta}\simeq {_{\alpha}
(R^*)_{\beta}}$, in particular $_{\alpha} (R^*)_{\beta}\simeq
{_1(R^*)_{\alpha^{-1}\beta}}$. Indeed, $_{\gamma\alpha}
(R^*)_{\gamma\beta}\simeq (_{\gamma\beta} R_{\gamma\alpha})^*\simeq
(_{\beta} R_{\alpha})^*\simeq {_{\alpha} (R^*)_{\beta}}$.\\

$\bullet$ If $R$ has finite dimension, then $_{1}
(R^*)_{\alpha}\simeq {_{1} (R^*)_{\beta}}$ if and only if
$\alpha^{-1}\beta \in {\rm Inn}(R)$. Indeed, $_{1} (R^*)_{\alpha}$
and ${_{1} (R^*)_{\beta}}$ are isomorphic if and only if so are
their duals, i.e., $_\alpha  R_1\simeq {_\beta R_1}$, which is the
same to $_1R_{\alpha^{-1}} \simeq {_1R_{\beta^{-1}}}$, i.e.,
$\alpha^{-1}\beta \in {\rm Inn}(R)$. Since ${\rm Inn}(R)$ is a
normal subgroup of ${\rm Aut}(R)$, this is equivalent to
$\alpha\beta^{-1} \in {\rm Inn}(R)$.\\

The following holds for any finite-dimensional algebra.
\begin{proposition} \label{isoend}
Let $R$ be a finite-dimensional algebra. Then the map $\omega:R\ra
{\rm End}(_RR^*)$ defined by $\omega (a)(r^*)=r^*\lhu a$ for any
$r^*\in R^*$ and $a\in R$, is an isomorphism of algebras.
\end{proposition}
\begin{proof}
Since the linear dual functor is a duality between the categories of
finite-dimensional right, respectively left, $R$-modules, we have
$R\simeq {\rm End}(R_R)\simeq {\rm End}((R_R)^*)={\rm End}(_RR^*)$.
In particular $R$ and ${\rm End}(_RR^*)$ have the same dimension.

It is easy to check that $\omega$ is well defined and it is an
algebra morphism. If $\omega(a)=0$ for some $a$, then $r^*\lhu a=0$
for any $r^*\in R^*$, and evaluating at 1, we get $r^*(a)=0$. Thus
$a$ must be 0, so $\omega$ is injective, and then it is an
isomorphism.
\end{proof}

Let $R$ be a finite-dimensional algebra. We recall that  $R$ is
called quasi-Frobenius if it is injective as a left (or
equivalently, right) $R$-module. It is known that $R$ is
quasi-Frobenius if and only if the left $R$-modules $R$ and $R^*$
have the same distinct indecomposable components (possibly occurring
with different multiplicities), see \cite[Section 16C]{lam}.
Therefore a Frobenius algebra is always quasi-Frobenius.

\begin{corollary} \label{dualinvertible}
Let $R$ be a finite-dimensional algebra. Then $R^*$ is an invertible
$R$-bimodule if and only if $R$ is a quasi-Frobenius algebra.
\end{corollary}
\begin{proof}
If $R^*$ is an invertible bimodule, then it is projective as a right
$R$-module, so then its linear dual $(R^*)^*$ is an injective left
$R$-module. But $(R^*)^*\simeq R$ as left $R$-modules, and we get
that $R$ is left selfinjective.

Conversely, assume that $R$ is quasi-Frobenius. Since $R$ is an
injective right $R$-module, we get that $R^*$ is a projective left
$R$-module. On the other hand, since the left $R$-modules $R$ and
$R^*$ have the same distinct indecomposable components, we see that
there is an epimorphism $(R^*)^n\ra R$ for a large enough positive
integer $n$, thus $R^*$ is a generator as a left $R$-module. If we
also take into account Proposition \ref{isoend}, we get that $R^*$
is invertible.
\end{proof}

 If $R$ is Frobenius,
then an element $\lambda\in R^*$ such that $(R\rhu \lambda) =R^*$,
is called a Frobenius form on $R$; in this case, the map $a\mapsto
(a\rhu \lambda)$ is an isomorphism of left $R$-modules between $R$
and $R^*$, and also, the map $a\mapsto (\lambda\lhu a)$ is an
isomorphism of right $R$-modules from $R$ to $R^*$. The Nakayama
automorphism of $R$ associated with a Frobenius form $\lambda$ is
the map $\nu:R\ra R$ defined such that for any $a\in R$, $\nu (a)$
is the unique element of $R$ satisfying $(\nu(a)\rhu
\lambda)=(\lambda \lhu a)$, or equivalently, $\lambda (ar)=\lambda
(r\nu(a))$ for any $r\in R$; $\nu$ turns out to be an algebra
automorphism. If $\nu$ and $\nu'$ are Nakayama automorphisms
associated with two Frobenius forms, then there exists an invertible
element $u\in R$ such that $\nu'(a)=u^{-1}\nu (a)u$ for any $a\in
R$, thus $\nu$ and $\nu'$ are equal up to an inner automorphism. It
follows that the class of a Nakayama automorphism in ${\rm Out}(R)$
does not depend on the Frobenius form; see \cite[Section 16E]{lam},
\cite[Section 2.2]{lorenz} or \cite[Chapter IV]{sy} for details.

If the quasi-Frobenius algebra $R$ is not Frobenius, $R^*$ is not
isomorphic to any $_1R_\alpha$, as $R^*$ is not isomorphic to $R$ as
left $R$-modules. In the Frobenius case, we have the following
result; it appears in an equivalent formulation in \cite[Proposition
3.15]{sy}.

\begin{proposition} \label{dualalpha}
Let $R$ be a Frobenius algebra. Then there exists $\nu \in {\rm
Aut}(R)$ such that $R^*\simeq {_1R_\nu}$ as bimodules. Moreover, any
such $\nu$ is the Nakayama automorphism of $R$ associated with a
Frobenius form. As a consequence, the order of $[R^*]$ in ${\rm
Pic}(R)$ is equal to the order of the class of $\nu$ in ${\rm
Out}(R)$.
\end{proposition}
\begin{proof}
The first part follows directly from Proposition
\ref{isoinvertible}, since $R^*\simeq R$  as left $R$-modules.

Let $\gamma:{_1R_\nu}\ra R^*$ be an isomorphism of bimodules, and
let  $\lambda=\gamma (1)$. Then $(R\rhu \lambda) =R^*$, so $\lambda$
is a Frobenius form on $R$. Then for any $a,x\in R$ \bea(\lambda
\lhu a)(x)
&=&(\gamma (1)\lhu a)(x)\\
&=&\gamma (1\cdot \nu (a))(x)\\
&=&\gamma (\nu(a)\cdot 1)(x)\\
&=&(\nu(a)\rhu \gamma(1))(x)\\
&=&(\nu(a)\rhu \lambda)(x), \eea showing that $(\lambda \lhu
a)=(\nu(a)\rhu \lambda)$, thus $\nu$ is the Nakayama automorphism
associated with $\lambda$.
\end{proof}

Looking inside the Picard group,  the previous Proposition gives a
new perspective on the well-known fact that a Frobenius algebra is
symmetric if and only if the Nakayama automorphism is inner, see
\cite[Theorem 16.63]{lam}. Indeed, $R$ is symmetric if and only if
$R^*\simeq R$ as bimodules, i.e., ${_1R_\nu}\simeq R$, and this is
equivalent to $\nu$ being inner.

The following indicates a commutation property of the class of $R^*$
in the Picard group of $R$.

\begin{proposition} \label{dualulcomuta}
Let $R$ be a quasi-Frobenius finite-dimensional algebra, and let
$\alpha\in {\rm Aut}(R)$. Then $R^*\ot_R{_1R_\alpha}\simeq
{_1R_\alpha}\ot_RR^*$ as $R$-bimodules. Thus the element $[R^*]$ of
the Picard group ${\rm Pic}(R)$ lies in the centralizer of the image
of ${\rm Out}(R)$.
\end{proposition}
\begin{proof}
Taking into account the above considerations, we have isomorphisms
of $R$-bimodules
$$R^*\ot_R\,{_1R_\alpha}\simeq {_1(R^*)_\alpha}\simeq {_{\alpha^{-1}}(R^*)_1}\simeq {_{\alpha^{-1}}R_1}\ot_RR^*\simeq {_1R_\alpha}\ot_RR^*$$
\end{proof}

\begin{corollary}
Let $R$ be a Frobenius algebra. Then the class of the Nakayama
automorphism of $R$ lies in the centre of ${\rm Out}(R)$.
\end{corollary}

If $R$ is quasi-Frobenius, we are interested in the order of $[R^*]$
in the group ${\rm Pic}(R)$. This order is $1$ if and only if $R$ is
a symmetric algebra. The following examples show that  it may be any
integer $\geq 2$ in other quasi-Frobenius algebras, and also it can
be infinite.

For the first example, we recall that if $H$ is a finite-dimensional
Hopf algebra, then a left integral on $H$ is an element $\lambda \in
H^*$ such that $h^*\lambda =h^*(1)\lambda$ for any $h^*\in H^*$; the
multiplication of $H^*$ is given by the convolution product.  Any
finite-dimensional Hopf algebra $H$ is a Frobenius algebra, and a
non-zero left integral $\lambda$ on $H$ is a Frobenius form, see
\cite[Theorem 12.5]{lorenz}.

\begin{example} \label{exempleordinR^*}
{\rm Let $C$ be a finite abelian group, and let $C^*={\rm
Hom}(C,K^\times)$ be its character group. We consider certain Hopf
algebras in the braided category of Yetter-Drinfeld modules over the
group Hopf algebra $KC$, called quantum lines, and their liftings,
obtained by a bosonization construction. We obtain some
finite-dimensional pointed Hopf algebras with coradical $KC$, see
\cite{as}, \cite{bdg}. There are two classes of such objects.

$\bf{(I)}$ Hopf algebras of the type $H_1(C,n,c,c^*)$, where $n\geq
2$ is an integer, $c\in C$ and $c^*\in C^*$, such that $c^n\neq 1$,
$(c^*)^n=1$ and $c^*(c)$ is a primitive $n$th root of unity. It is
generated as an algebra by the Hopf subalgebra $KC$ and a
$(1,c)$-skewprimitive element $x$, i.e., the comultiplication works
as $\Delta (x)=c\ot x+x\ot 1$ on $x$, subject to relations
$x^n=c^n-1$ and $xg=c^*(g)gx$ for any $g\in C$. Note that the
required conditions show that $c^*$ has order $n$.

$\bf{(II)}$ Hopf algebras of the type $H_2(C,n,c,c^*)$, where $n\geq
2$ is an integer, $c\in C$ and $c^*\in C^*$, such that $c^*(c)$ is a
primitive $n$th root of unity. It is generated as an algebra by the
Hopf subalgebra $KC$ and a $(1,c)$-skewprimitive element $x$,
subject to relations $x^n=0$ and $xg=c^*(g)gx$ for any $g\in C$. In
this case, the order of $c^*$, which we denote by $m$, is a multiple
of $n$.

If $H$ is any of $H_1(C,n,c,c^*)$ or $H_2(C,n,c,c^*)$, a linear
basis of $H$ is ${\mathcal{B}}=\{ gx^j|g\in C, 0\leq j\leq n-1\}$,
thus the dimension of $H$ is $n|C|$, and the linear map $\lambda\in
H^*$ such that $\lambda(c^{1-n}x^{n-1})=1$ and $\lambda$ takes any
other element of $\mathcal{B}$ to 0, is a left integral on $H$, see
\cite[Proposition 1.17]{bdg}.

If $g\in C$, then
$$(\lambda \lhu g)(g^{-1}c^{1-n}x^{n-1})=\lambda(c^{1-n}x^{n-1})=1$$
and $\lambda\lhu g$ takes any other element of $\mathcal{B}$ to 0,
while
$$(g\rhu
\lambda)(g^{-1}c^{1-n}x^{n-1})=\lambda(g^{-1}c^{1-n}x^{n-1}g)=c^*(g)^{n-1}\lambda(c^{1-n}x^{n-1})=c^*(g)^{n-1}$$
and $g\rhu \lambda$ takes any other element of $\mathcal{B}$ to 0.
These show that $(g\rhu \lambda) =c^*(g)^{n-1}(\lambda\lhu g)$, so
the Nakayama automorphism $\nu$ associated with the Frobenius form
$\lambda$ satisfies $\nu(g)=c^*(g)^{1-n}g$.

On the other hand, if we denote $\xi=c^*(c)$, we have
$$(x\rhu \lambda)(c^{1-n}x^{n-2})=\lambda(c^{1-n}x^{n-1})=1$$
and $x\rhu \lambda$ takes any other element of $\mathcal{B}$ to 0,
while
$$(\lambda \lhu
x)(c^{1-n}x^{n-2})=\lambda(xc^{1-n}x^{n-2})=\xi^{1-n}\lambda
(c^{1-n}x^{n-1})=\xi$$ and $\lambda\lhu x$ takes any other element
of $\mathcal{B}$ to 0. Thus we get $\nu(x)=\xi x$.

Denote the order of $c^*$ by $m$; we noticed that $m=n$ in the case
of $H_1(C,n,c,c^*)$, and $m=dn$ for some positive integer $d$ in the
case of $H_2(C,n,c,c^*)$. If $j$ is a positive integer, then
$\nu^j=1$ if and only if $\xi^j=1$ and $c^*(g)^{j(1-n)}=1$ for any
$g\in C$. If the latter condition is satisfied, then
$(c^*)^{j(1-n)}=1$, or equivalently, $m|j(1-n)$, hence $n|j(1-n)$,
and then $n|j$, so the condition $\xi^j=1$ is automatically
satisfied. Thus the order of $\nu$ is the least positive integer $j$
such that $m|j(1-n)$. For any such $j$ we have $n|j$, so  $j=bn$ for
some integer $b$. Then $m|j(1-n)$ is equivalent to $d|b(n-1)$, and
also to $\frac{d}{(d,n-1)}|b\cdot \frac{n-1}{(d,n-1)}$. Since
$\frac{d}{(d,n-1)}$ and $\frac{n-1}{(d,n-1)}$ are relatively prime,
the latter condition is equivalent to $\frac{d}{(d,n-1)}|b$. We
conclude that the least such $b$ is $\frac{d}{(d,n-1)}$, and the
order of $\nu$ is
$$j=bn=\frac{dn}{(d,n-1)}=\frac{m}{(\frac{m}{n},n-1)}.$$

This shows that for $H_1(C,n,c,c^*)$, where $m=n$, the order of
$\nu$ is necessarily $n$, while for $H_2(C,n,c,c^*)$, the order may
be larger than $n$, depending on the value of $m$.

Now we show that for any $1\leq j<\frac{m}{(\frac{m}{n},n-1)}$,
$\nu^j$ is not an inner automorphism. Indeed, if it were, then there
would exist an invertible $u$ such that $\nu^j(r)=u^{-1}ru$ for any
$r$ in the Hopf algebra (which is either $H_1(C,n,c,c^*)$ or
$H_2(C,n,c,c^*)$). In particular, for any $g\in C$,
$c^*(g)^{j(1-n)}g=u^{-1}gu$. Applying the counit $\varepsilon$, one
gets $c^*(g)^{j(1-n)}=1$ for any $g\in C$, so $(c^*)^{j(1-n)}=1$.
Hence $m|j(1-n)$, and we have seen above that this implies that $j$
must be at least $\frac{m}{(\frac{m}{n},n-1)}$, a contradiction.

We conclude that if $A$ is a Hopf algebra of type $H_1(C,n,c,c^*)$
or $H_2(C,n,c,c^*)$, then the order of the Nakayama automorphism
$\nu$ of $A$ in the group of algebra automorphisms of $A$, as well
as the order of the class of $\nu$ in ${\rm Out}(A)$ (which is the
same with the order of $[A^*]$ in ${\rm Pic}(A)$) is
$\frac{m}{(\frac{m}{n},n-1)}$, where $m$ is the order of $c^*$ in
$C^*$. In the case of $H_1(C,n,c,c^*)$, where $m=n$, this order is
just $n$.

A particular case is when $C=C_n=<c>$ is the cyclic group of order
$n\geq 2$. Then for any linear character $c^*\in C^*$ such that
$c^*(c)$ is a primitive $n$th root of unity, $H_1(C,n,c,c^*)$ is a
Taft Hopf algebra. For such algebras, the order of the Nakayama
automorphism associated with a left integral as a Frobenius form is
computed in \cite[Example 5.9, page 614]{sy}.}
\end{example}

\begin{example} \label{exemplequantumplane}
{\rm  Let $q$ be a non-zero element of a field $K$, and let
$K_q[X,Y]$ be the quantum plane, which is the $K$-algebra generated
by $X$ and $Y$, subject to the relation $YX=qXY$. Let
$R_q=K_q[X,Y]/(X^2,Y^2)$, which has dimension 4, and a basis
$\mathcal{B}=\{ 1,x,y,xy\}$, where $x,y$ denote the classes of $X,Y$
in $R$. We have $x^2=y^2=0$ and $yx=qxy$. Denote by
$\mathcal{B}^*=\{ 1^*,x^*,y^*,(xy)^*\}$ the basis of $R_q^*$ dual to
$\mathcal{B}$. Then
$$1\rhu (xy)^*=(xy)^*, x\rhu (xy)^*=qy^*, y\rhu (xy)^*=x^*, (xy)\rhu (xy)^*=1^*,$$  showing that the
linear map from $R_q$ to $R_q^*$ which takes $r$ to $r\rhu (xy)^*$
is an isomorphism. Thus $R_q$ is a Frobenius algebra and
$\lambda=(xy)^*$ is a Frobenius form on $R_q$. Now since $(xy)^*\lhu
x=y^*$ and $(xy)^*\lhu y=qx^*$, the Nakayama automorphism associated
with $\lambda$ is $\nu\in {\rm Aut}(R_q)$ given by $\nu(x)=q^{-1}x,
\nu(y)=qy$. Then it is clear that the order of $\nu$ in the
automorphism group of $R_q$ is $n$ if $q$ is a primitive $n$th root
of unity in $K$, and it is infinite when no non-trivial power of $q$
is 1. This fact was observed in \cite[Example 10.7, page 417]{sy} by
using periodic modules with respect to actions of the syzygy and
Auslander-Reiten operators.

We show that if $t$ is a positive integer such that $q^t\neq 1$,
then $\nu^t$ is not even an inner automorphism. Indeed, if it were,
then $\nu^t(x)=u^{-1}xu$, or $ux=q^txu$ for some invertible $u\in
R_q$. If we write $u=a1+bx+cy+dxy$ with $a,b,c,d\in K$, this  means
that $ax+qcxy=q^t(ax+cxy)$, showing that $a=0$. But then $u$ cannot
be invertible, since $xyu=0$, a contradiction.

In conclusion, if $q$ is not a root of unity, $\nu$ has infinite
order in ${\rm Out}(R_q)$, and so does $[R_q^*]$ in ${\rm
Pic}(R_q)$, while if $q$ is a primitive $n$th root of unity, then
$[R_q^*]$ has order $n$ in ${\rm Pic}(R_q)$.

We end this example with the remark that Nakayama and Nesbitt
constructed in \cite[page 665]{nn} a class of examples of Frobenius
algebras which are not symmetric, presented in a matrix form. More
precisely, in the presentation of \cite[Example 16.66]{lam}, for any
non-zero elements $u,v\in K$, let $A_{u,v}$ be the subalgebra of
$M_4(K)$ consisting of all matrices of the
type $\left[ \begin{array}{cccc} a & b& c&d\\
0 & a & 0&uc\\
0&0&a&vb\\
0&0&0&a\end{array} \right]$, where $a,b,c,d\in K$. Then $A_{u,v}$ is
Frobenius for any $u,v\in K^{\times}$, and it is symmetric if and
only if $u=v$.  $A_{u,v}$ has a basis consisting of the elements
$$I_4,\; x=E_{12}+vE_{34},\; y=E_{13}+uE_{24},\; z=E_{14},$$
where $E_{ij}$ denote the usual matrix units in $M_4(K)$, and they
satisfy the relations
$$x^2=0,\; y^2=0,\; xy=uz,\; yx=vz.$$
These show that in fact, $A_{u,v}$ is isomorphic to the quotient
$R_{u^{-1}v}$ of the quantum plane. }
\end{example}

If $H$ is a finite-dimensional Hopf algebra, let $t\in H$ be a
non-zero left integral in $H$, i.e., $ht=\varepsilon (h)t$ for any
$h\in H$, where $\varepsilon$ is the counit of $H$. As the space of
left integrals is one-dimensional and $th$ is a left integral for
any $h\in H$, there is a linear map $\mathcal{G}:H\ra K$ such that
$th=\mathcal{G}(h)t$ for any $h\in H$. In fact, $\mathcal{G}$ is an
algebra morphism, thus an element of the group $G(H^*)$ of grouplike
elements of $H^*$. $\mathcal{G}$ is called the distinguished
group-like element of $H^*$, and also the right modular element of
$H^*$.

\begin{theorem} \label{ordinH*}
Let $H$ be a finite-dimensional Hopf algebra with antipode $S$ and
counit $\varepsilon$, and let $\mathcal{G}$ be the modular element
in $H^*$. If $n$ is a positive integer, then $[H^*]^n=1$ in ${\rm
Pic}(H)$ if and only if $S^{2n}$ is inner and
$\mathcal{G}^n=\varepsilon$. As a consequence, the order of $[H^*]$
in ${\rm Pic}(H)$ is the least common multiple of the order of the
class of $S^2$ in ${\rm Out}(H)$ and the order of $\mathcal{G}$ in
$G(H^*)$.
\end{theorem}
\begin{proof}
Let $\lambda$ be a non-zero left integral on $H$, which is a
Frobenius form on $H$, and let $\nu$ be the associated Nakayama
automorphism. By \cite[Theorem 3(a)]{radford}, in the reformulation
of \cite[Proposition 12.8]{lorenz},  $\nu(h)=\sum
\mathcal{G}(h_2)S^2(h_1)$ for any $h\in H$. Let
$\ell_{\mathcal{G}}:H\ra H$ be the linear map defined by
$\ell_{\mathcal{G}}(h)=\mathcal{G}\rhu h=\sum \mathcal{G}(h_2)h_1$.
We have $\nu=S^2 \ell_{\mathcal{G}}$. We note that
$\mathcal{G}S^2=\mathcal{G}$. Indeed, it is clear that
$\mathcal{G}S=S^*(\mathcal{G})=\mathcal{G}^{-1}$, since the dual map
$S^*$ of $S$ is the antipode of the dual Hopf algebra $H^*$, and it
takes a group-like element to its inverse. Now we have \bea
(\ell_{\mathcal{G}}S^2)(h)&=&\mathcal{G}\rhu S^2(h)\\
&=&\sum \mathcal{G}(S^2(h_2))S^2(h_1)\\
&=&\sum \mathcal{G}(h_2)S^2(h_1)\\
&=&S^2(\mathcal{G}\rhu h)\\
&=&(S^2\ell_{\mathcal{G}})(h), \eea showing that
$\ell_{\mathcal{G}}S^2=S^2\ell_{\mathcal{G}}$. Since $\rhu$ is a
left action, we have $(\ell_{\mathcal{G}})^n=\ell_{\mathcal{G}^n}$
for any positive integer $n$, and it follows that
$\nu^n=S^{2n}\ell_{\mathcal{G}^n}$. Now if
$\ell_{\mathcal{G}^n}=\varepsilon$ and  $S^{2n}$ is inner, then
$\nu^n=S^{2n}$ is inner, so
$[H^*]^n=[{_1H_\nu}]^n=[{_1H_{\nu^n}}]=1$ in ${\rm Pic}(H)$.
Conversely, if $[H^*]^n=1$, then $\nu^n$ is inner. Let
$\nu^n(h)=u^{-1}hu$ for some invertible $u\in H$. Then
$S^{2n}(\ell_{\mathcal{G}^n}(h))=u^{-1}hu$ for any $h\in H$, and
applying $\varepsilon$  and using that $\varepsilon S=\varepsilon$,
we obtain $\varepsilon
(\ell_{\mathcal{G}^n}(h))=\varepsilon(u^{-1})\varepsilon(h)\varepsilon(h)=\varepsilon(h)$.
As $\varepsilon (\ell_{\mathcal{G}^n}(h))=\varepsilon (\sum
\mathcal{G}^n(h_2)h_1)=\mathcal{G}^n(h)$, we get
$\mathcal{G}^n=\varepsilon$. Consequently, $\nu^n=S^{2n}$, so
$S^{2n}$ is inner.
\end{proof}

We note that in the particular case where $n=1$, the previous
Theorem says that a finite-dimensional Hopf algebra $H$ is a
symmetric algebra if and only if $\mathcal{G}=\varepsilon$, i.e.,
$H$ is unimodular, and $S^2$ is inner. This is a result of
\cite{os}, see also \cite[Theorem 12.9]{lorenz}.

\section{Picard groups of quasi-Frobenius algebras}
\label{sectionpicardquasi}

We start with some general considerations. Let $R$ and $S$ be two
Morita equivalent $K$-algebras. Let
$(R,S,{_RP_S},{_SQ_R},P\ot_SQ\stackrel{f}{\ra}
R,Q\ot_RP\stackrel{g}{\ra} S)$ be a strict Morita context connecting
$R$ and $S$, i.e., $P$ and $Q$ are bimodules as the indices
indicate, $f$ is an isomorphism of $R$-bimodules, $g$ is an
isomorphism of $S$-bimodules, and denoting $f(p\ot_Sq)=[p,q]$ and
$g(q\ot_Rp)=(q,p)$, the conditions $[p,q]p'=p(q,p')$ and
$(q,p)q'=q[p,q']$ hold for any $p,p'\in P, q,q'\in Q$.

As explained in \cite[pages 301-302]{takeuchi}, a $K$-linear
monoidal equivalence $F=Q\ot_R (-)\ot_RP$ is induced between the
monoidal categories of $R$-bimodules and $S$-bimodules, with
quasi-inverse $G=P\ot_S(-)\ot_SQ$.
\begin{proposition} \label{3.1}
The mapping $[M]\mapsto [F(M)]$ is an isomorphism between the groups
${\rm Pic}(R)$ and ${\rm Pic}(S)$. In particular, the order of $[M]$
in ${\rm Pic}(R)$ is equal to the order of $[Q\ot_RM\ot_RP]$ in
${\rm Pic}(S)$.
\end{proposition}
\begin{proof}
Since $F$ is a monoidal equivalence, we see that if $M$ is an
invertible $R$-bimodule, then $F(M)$ is an invertible $S$-bimodule,
and moreover, the mapping $[M]\mapsto [F(M)]$ is a group morphism.
Its inverse takes $[N]$ to $[G(N)]$ for any invertible $S$-bimodule
$N$.
\end{proof}

\begin{proposition} \label{3.2}
Assume that $R$ and $S$ are Morita equivalent finite-dimensional
$K$-algebras, and let $F$ be the monoidal equivalence between the
categories of $R$-bimodules and $S$-bimodules described above. Then
there is an isomorphism of $S$-bimodules $F(R^*)\simeq S^*$. In
particular, the order of $[R^*]$ in ${\rm Pic}(R)$ is equal to the
order of $[S^*]$ in ${\rm Pic}(S)$.
\end{proposition}
\begin{proof}
We first list some basic facts. We denote by $_A\mathcal{M}_B$ the
category of left $A$, right $B$-bimodules, $\mathrm{Hom}_{A-}$ means
morphisms of left $A$-modules, while $\mathrm{Hom}_{-A}$ means
morphisms of right $A$-modules. (i) below is the tensor-Hom
adjunction, and (ii) is the duality property, where the structure of
the involved objects is enriched to bimodules.
(iii) is basic Morita theory, and so is (iv), with the mention that the isomorphism is also of $S$-bimodules.
\\[2mm]
(i) If $M\in {_A\mathcal{M}_B}$ and $N\in {_B\mathcal{M}_C}$ are
finite dimensional, then $(M\ot _BN)^*\simeq
\mathrm{Hom}_{-B}(M,N^*)$ in
$_C\mathcal{M}_A$.\\
(ii) If $U\in {_A\mathcal{M}_B}$ and $V\in {_A\mathcal{M}_C}$ are
finite dimensional, then $\mathrm{Hom}_{A-}(U,V)\simeq
\mathrm{Hom}_{-A}(V^*,U^*)$ in $_B\mathcal{M}_C$.\\
(iii) $\mathrm{Hom}_{R-}(P,R)\simeq Q$ in $_S\mathcal{M}_R$.\\
(iv) $\mathrm{Hom}_{-R}(Q,Q)\simeq S$ in $_S\mathcal{M}_S$.\\

Using these, we have the following isomorphisms of $S$-bimodules
\bea (Q\ot_RR^*\ot_RP)^*&\simeq & \mathrm{Hom}_{-R}(Q,(R^*\ot_R
P)^*) \;\;\;\;\;\;\;\;\; (\mathrm{by}\; \mathrm{(i)}\;
\mathrm{for}\;
M={_SQ_R}, N={_R(R^*\ot_RP)_S})\\
&\simeq& \mathrm{Hom}_{-R}(Q,\mathrm{Hom}_{-R}(R^*,P^*))\;\;
(\mathrm{by}\; \mathrm{(i)}\; \mathrm{for}\; M={_RR^*_R},
N={_RP_S})\\
&\simeq& \mathrm{Hom}_{-R}(Q,\mathrm{Hom}_{R-}(P,R))\;\;\;\;\;\;
(\mathrm{by}\; \mathrm{(ii)}\; \mathrm{for}\; U={_RP_S}, V={_RR_R})\\
&\simeq & \mathrm{Hom}_{-R}(Q,Q)
\;\;\;\;\;\;\;\;\;\;\;\;\;\;\;\;\;\;\;\;\;\;\;\;\; (\mathrm{by}\;
\mathrm{(iii)})\\
&\simeq&S
\;\;\;\;\;\;\;\;\;\;\;\;\;\;\;\;\;\;\;\;\;\;\;\;\;\;\;\;\;\;\;\;\;\;\;\;\;\;\;\;\;\;\;\;
(\mathrm{by}\; \mathrm{(iv)})\eea Taking duals, we find that
$F(R^*)=Q\ot_RR^*\ot_RP\simeq S^*$ as $S$-bimodules. \end{proof}

\begin{remark}
One of the referees indicated us an alternative way for proving this
by using basic results on the Nakayama functor. Thus if we denote by
$T=Q\ot_R(-)$ the equivalence functor between the categories of left
$R$-modules and left $S$-modules, and by $L=P\ot_S(-)$ its
quasi-inverse, we have isomorphisms
$$F(R^*)\ot_S(-)\simeq T\circ
\mathrm{N}^\mathrm{r}_{R-\mathrm{mod}}\circ L\simeq T\circ L\circ
\mathrm{N}^\mathrm{r}_{S-\mathrm{mod}}\simeq S^*\ot_S(-)$$ in the
category of right exact linear functors from $S-\mathrm{mod}$ to
$R-\mathrm{mod}$ (see \cite[Lemma 3.15 and Theorem 3.18]{fss}),
where $\mathrm{N}^\mathrm{r}$ denotes the right exact Nakayama
functor introduced in \cite{fss}. It follows that $F(R^*)$ and $S^*$
are isomorphic.
\end{remark}

Now if $R$ is a finite-dimensional quasi-Frobenius algebra, we
consider a basic algebra $S$ of $R$. As explained in \cite[page
172]{sy}, $S$ can be constructed as follows. Take a complete system
of orthogonal primitive idempotents in $R$, and let $e$ be the sum
of a system of representatives of the isomorphism types of the
idempotents in this system. Then $S=eRe$ is a basic algebra of $R$;
it is Frobenius and Morita equivalent to $R$, see \cite[Theorem
6.16, page 173, and Corollary 3.11, page 351]{sy}. A basic algebra
of $R$ is not uniquely determined, but it is unique up to an
isomorphism. As a consequence of the above discussion, we obtain:

\begin{corollary} \label{3.4}
If $R$ is a finite-dimensional quasi-Frobenius algebra, and $S$ is a
basic algebra of $R$, then ${\rm Pic}(R)\simeq {\rm Pic}(S)$ and the
order of $[R^*]$ in ${\rm Pic}(R)$ is equal to the order of the
class of the Nakayama automorphism of $S$ (with respect to some
Frobenius form) in ${\rm Out}(S)$.
\end{corollary}

\section{The structure of $\mathcal{R}$ and $\mathcal{R}^*$, and the Picard group of $\mathcal{R}$} \label{sectionconstruction}

Let $\mathcal{R}$ be the $K$-algebra presented in the Introduction.
It has basis ${\bf B}=\{ E, X_1,X_2,Y_1,Y_2\}\cup \{F_{ij}|1\leq
i,j\leq 2\}$, and relations

 \bea E^2=E,&
F_{ij}F_{jr}=F_{ir},\\  EX_{i}=X_{i},& X_{i}F_{ir}=X_{r},\\
F_{ij}Y_{j}=Y_{i},& Y_{i}E=Y_{i} \eea for any $1\leq i,j,r\leq 2$,
 and any other product of two elements of $\bf B$ is zero.

Let \bea \mathcal{V}_1&=&<X_1,F_{11},F_{21}>\\
\mathcal{V}'_1&=&<X_2,F_{12},F_{22}>\\
\mathcal{V}_2&=&<Y_1,Y_2,E>. \eea Then
$\mathcal{R}=\mathcal{V}_1\oplus \mathcal{V}'_1\oplus \mathcal{V}_2$
is a decomposition of $\mathcal{R}$ into a direct sum of
indecomposable left $\mathcal{R}$-modules, and $\mathcal{V}_1\simeq
\mathcal{V}'_1\simeq{\hspace{-3.5mm}/}$$\hspace{2mm} \mathcal{V}_2$.
Indeed, right multiplication by $F_{12}$ is an isomorphism from
$\mathcal{V}_1$ to $\mathcal{V}'_1$, with inverse the right
multiplication by $F_{21}$, while  $\mathcal{V}_1$ and
$\mathcal{V}_2$ are not isomorphic since they have different
annihilators.

Similarly, a decomposition of $\mathcal{R}$ into a direct sum of
indecomposable right $\mathcal{R}$-modules is
$\mathcal{R}=\mathcal{U}_1\oplus \mathcal{U}_2\oplus
\mathcal{U}'_2$, with $\mathcal{U}_2\simeq
\mathcal{U}'_2\simeq{\hspace{-3.5mm}/}$$\hspace{2mm} \mathcal{U}_1$,
where
\bea \mathcal{U}_1&=&<E,X_1,X_2>\\
\mathcal{U}_2&=&<F_{11},F_{12},Y_1>\\
\mathcal{U}'_2&=&<F_{21},F_{22},Y_2>. \eea

The quotient algebra of $\mathcal{R}$ by the nilpotent ideal
$<X_1,X_2,Y_1,Y_2>$ is isomorphic to $K\times M_2(K)$, so the
Jacobson radical $J(\mathcal{R})=<X_1,X_2,Y_1,Y_2>$. Then
$\mathcal{U}_1J(\mathcal{R})=<X_1,X_2>$ and
$\mathcal{U}_2J(\mathcal{R})=<Y_1>$, so the isomorphism types of
simple right $\mathcal{R}$-modules are
$S_1=\mathcal{U}_1/\mathcal{U}_1J(\mathcal{R})\simeq
<Y_1>=\mathrm{soc}(\mathcal{U}_2)$ and
$S_2=\mathcal{U}_2/\mathcal{U}_2J(\mathcal{R})\simeq
<X_1,X_2>=\mathrm{soc}(\mathcal{U}_1)$, where $\mathrm{soc}(U)$
denotes the socle of the module $U$. These show that the
multiplicity $\mu (S_i,\mathcal{U}_j)$ of the simple module $S_i$ in
a composition series of $\mathcal{U}_j$ is 1 for any $1\leq i,j\leq
2$. Clearly, $\mathrm{End}_\mathcal{R}(S_i)\simeq K$ for each $i$.

Now we look at $\mathcal{R}^*=Hom_K(\mathcal{R},K)$, with the
$\mathcal{R}$-bimodule structure induced by the one of
$\mathcal{R}$; we denote by $\rhu$ and $\lhu$ the left and right
actions of $\mathcal{R}$ on $\mathcal{R}^*$. Denote by ${\bf B}^*=\{
E^*, F_{ij}^*,X_i^*,Y_j^*| 1\leq i,j\leq 2\}$ the basis of
$\mathcal{R}^*$ dual to $\bf B$. On basis elements, the left action
of $\mathcal{R}$ on $\mathcal{R}^*$ is

$$\begin{array}{llll} E\rhu E^*=E^*,& E\rhu F_{ij}^*=0,& E\rhu X_i^*=0,& E\rightarrow
Y_i^*=Y_i^*,\\
F_{ij}\rhu E^*=0,& F_{ij}\rhu F_{rp}^*=\delta_{jp}F_{ri}^*,&
F_{ij}\rhu
X_r^*=\delta_{jr}X_i^*,& F_{ij}\rhu Y_r^*=0,\\
X_i\rhu E^*=0,& X_i\rhu F_{rj}^*=0,& X_i\rhu X_j^*=\delta_{ij}E^*,&
X_i\rhu Y_j^*=0,\\
Y_i\rhu E^*=0,& Y_i\rhu F_{rj}^*=0,& Y_i\rhu X_j^*=0,& Y_i\rhu
Y_j^*=F_{ji}^*,
 \end{array}$$
for any $1\leq i,j,r,p\leq 2$.

We will identify $\mathcal{U}_1^*$ with $<E^*,X_1^*,X_2^*>$ inside
$\mathcal{R}^*$, and similarly for the duals of $\mathcal{U}_2,
\mathcal{U}'_2, \mathcal{V}_1, \mathcal{V}'_1, \mathcal{V}_2$.

\begin{lemma} \label{lemaduale}
$\mathcal{U}_1^*\simeq \mathcal{V}_1$ and  $\mathcal{U}_2^*\simeq
\mathcal{V}_2$ as left $\mathcal{R}$-modules. Consequently,
$\mathcal{V}_1^*\simeq \mathcal{U}_1$ and  $\mathcal{V}_2^*\simeq
\mathcal{U}_2$ as right $\mathcal{R}$-modules, $\mathcal{R}^*\simeq
\mathcal{V}_1\oplus \mathcal{V}_2^2$ as left $\mathcal{R}$-modules
and $\mathcal{R}^*\simeq \mathcal{U}_1^2\oplus \mathcal{U}_2$ as
right $\mathcal{R}$-modules.
\end{lemma}
\begin{proof}
It follows from the action table above that the linear map taking
$X_1$ to $E^*$, $F_{11}$ to $X_1^*$ and $F_{21}$ to $X_2^*$ is an
isomorphism of left $\mathcal{R}$-modules from $\mathcal{V}_1$ to
$\mathcal{U}_1^*$. Also, the mapping $Y_1\mapsto F_{11}^*$,
$Y_2\mapsto F_{12}^*$, $E\mapsto Y_1^*$ defines an isomorphism
$\mathcal{V}_2\simeq \mathcal{U}_2^*$.
\end{proof}

The proof of the following Corollary was suggested by one of the
referees. Our initial proof was more computational.

\begin{corollary} \label{propdimensiuni}
${\rm dim}_K\, (\mathcal{U}_i\otimes_\mathcal{R} \mathcal{V}_j)=1$
for any $1\leq i,j\leq 2$.
\end{corollary}
\begin{proof}
By the tensor-Hom adjunction and taking into account Lemma
\ref{lemaduale}, we have linear isomorphisms
$$(\mathcal{U}_i\otimes_\mathcal{R} \mathcal{V}_j)^*=\mathrm{Hom}_K(\mathcal{U}_i\otimes_\mathcal{R}
\mathcal{V}_j,K)\simeq
\mathrm{Hom}_\mathcal{R}(\mathcal{U}_i,\mathcal{V}_j^*)\simeq
\mathrm{Hom}_\mathcal{R}(\mathcal{U}_i,\mathcal{U}_j)$$ for any
$i,j$. Now the result follows since $\mathcal{U}_i$ is a projective
cover of $S_i$, and we have by \cite[Proposition 2.8]{lorenz} that
$$\mathrm{dim}_K(\mathrm{Hom}_\mathcal{R}(\mathcal{U}_i,\mathcal{U}_j))=\mu(S_i,\mathcal{U}_j)\mathrm{dim}_K(\mathrm{End}_\mathcal{R}(S_i))=1.$$

\end{proof}

\begin{remark} \label{remarcadimensiuni}
{\rm The only non-zero tensor monomials formed with elements of $\bf
B$ in tensor products of the form $\mathcal{U}_i\otimes_\mathcal{R}
\mathcal{V}_j$ are: $E\ot _\mathcal{R}X_1=X_1\ot
_\mathcal{R}F_{11}=X_2\ot _\mathcal{R}F_{21}$ in
$\mathcal{U}_1\otimes_\mathcal{R} \mathcal{V}_1$,  $F_{11}\ot
_\mathcal{R}F_{11}=F_{12}\ot _\mathcal{R}F_{21}$ in
$\mathcal{U}_2\otimes_\mathcal{R} \mathcal{V}_1$,  $Y_1\ot
_\mathcal{R}E=F_{11}\ot _\mathcal{R}Y_1=F_{12}\ot _\mathcal{R}Y_2$
in $\mathcal{U}_2\otimes_\mathcal{R} \mathcal{V}_2$, and  $E\ot
_\mathcal{R}E$ in $\mathcal{U}_1\otimes_\mathcal{R} \mathcal{V}_2$.
Indeed, it is straightforward to check that any other such tensor
monomial in some $\mathcal{U}_i\otimes_\mathcal{R} \mathcal{V}_j$ is
zero. Then the above mentioned tensor monomials must be non-zero
since each $\mathcal{U}_i\otimes_\mathcal{R} \mathcal{V}_j$ has
dimension 1.}
\end{remark}

Now let $\mathcal{S}=e\mathcal{R}e$ be a basic algebra of
$\mathcal{R}$, where $e=E+F_{11}$, so then $\mathcal{S}=<E,F_{11},
X_1,Y_1>$.

\begin{proposition}
The order of the class of a Nakayama automorphism of $\mathcal{S}$
in ${\rm Out}(\mathcal{S})$ is $2$.
\end{proposition}
\begin{proof}
For the simplicity of notation, we renote just for this proof
$F_{11}=F, X_1=X,Y_1=Y$. Also denote by $
\mathcal{B}^*=\{E^*,F^*,X^*,Y^*\}$ the dual basis of
$\mathcal{B}=\{E,F,X,Y\}$ in $\mathcal{S}^*$. A direct computation
shows that the only non-zero elements of the form $b\ra b^*$, where
$b\in \mathcal{B}$, $b^*\in \mathcal{B}^*$, and $\ra$ denotes the
left action of $\mathcal{S}$ on $\mathcal{S}^*$, are
$$E\ra E^*=E^*, F\ra F^*=F^*, F\ra X^*=X^*, X\ra X^*=E^*,E\ra
Y^*=Y^*,Y\ra Y^*=F^*,$$ while the only non-zero elements of the form
$b^*\leftarrow b$, with $b\in \mathcal{B}$, $b^*\in \mathcal{B}^*$,
are
$$E^*\leftarrow E=E^*, F^*\leftarrow F=F^*, X^*\leftarrow
E=X^*,X^*\leftarrow X=F^*, Y^*\leftarrow F=Y^*,Y^*\leftarrow
Y=E^*.$$ Using these relations, we see that $\lambda=X^*+Y^*$ is a
Frobenius form of $\mathcal{S}$, and the corresponding Nakayama
automorphism is $\nu:\mathcal{S}\ra \mathcal{S}$ given by
$$\nu(E)=F,\nu(F)=E,\nu(X)=Y,\nu(Y)=X,$$
thus $\nu^2=Id$. On the other hand, $\nu$ is not an inner
automorphism, otherwise $\mathcal{S}$ would be a symmetric algebra,
and then so would be $\mathcal{R}$, as it is Morita equivalent to
$\mathcal{S}$; this is a contradiction. Alternatively, a direct
simple argument can be given to show that $\nu$ is not inner.
Indeed, if $u\in \mathcal{S}$ would be invertible such that
$u^{-1}Eu=F$, then writing $u=\alpha E+\beta F+\gamma X+\delta Y$
for some scalars $\alpha,\beta,\gamma,\delta$, we derive from
$Eu=uF$ that $\alpha=\beta=0$, and then $u^2=(\gamma X+\delta
Y)^2=0$, so $u$ could not be invertible.
\end{proof}

\begin{corollary} \label{orderdualNakayama}
The order of $\mathcal{R}^*$ in ${\rm Pic}(\mathcal{R})$ is 2, thus
there is an isomorphism of $\mathcal{R}$-bimodules
$\varphi:\mathcal{R}^*\ot_\mathcal{R}\mathcal{R}^*\ra \mathcal{R}$.
\end{corollary}
\begin{proof}
It follows from Corollary \ref{3.4} and the previous Proposition.
\end{proof}

\begin{remark} \label{remarcavarphi}
It is possible to obtain a direct computational proof of the
previous Corollary, and an explicit isomorphism
$\varphi:\mathcal{R}^*\ot_\mathcal{R}\mathcal{R}^*\ra \mathcal{R}$,
 by taking
into account Lemma \ref{lemaduale}, Remark \ref{remarcadimensiuni}
the table with the left $\mathcal{R}$-action on $\mathcal{R}^*$, and
a similar one with the right action. One can obtain a basis of
$\mathcal{R}^*\ot _\mathcal{R}\mathcal{R}^*$  consisting of the
elements $$ \mathcal{E}=Y_1^*\ot _\mathcal{R}X_1^*=Y_2^*\ot
_\mathcal{R}X_2^*,  \mathcal{F}_{ij}=X_i^*\ot _\mathcal{R}Y_j^*,
\mathcal{X}_i=E^*\ot _\mathcal{R}Y_i^*, \mathcal{Y}_i=F_{1i}^*\ot
_\mathcal{R}X_1^*,$$  where $1\leq i,j\leq 2$.

Moreover, the linear map $\varphi:\mathcal{R}^*\ot
_\mathcal{R}\mathcal{R}^*\ra \mathcal{R}$ given by $\varphi
(\mathcal{E})=E, \, \varphi (\mathcal{F}_{ij})=F_{ij},\,
\varphi(\mathcal{X}_i)=X_i,\, \varphi(\mathcal{Y}_i)=Y_i$ for any
$1\leq i,j\leq 2$, is an isomorphism of $\mathcal{R}$-bimodules.
\end{remark}

We aim to compute the Picard group of $\mathcal{R}$. One possibility
is to compute first the Picard group of the Frobenius algebra
$\mathcal{S}$, and then to use Corollary \ref{3.4}. As determining
the invertible bimodules and the group of exterior automorphisms of
$\mathcal{S}$ is of comparable difficulty to determining the ones of
$\mathcal{R}$, we prefer to look directly at the Picard group of
$\mathcal{R}$. This will also make the description more explicit.

\begin{lemma} \label{invertibleleftmodule}
Let $P$ be an invertible $\mathcal{R}$-bimodule. Then $P$ is
isomorphic either to $\mathcal{V}_1\oplus \mathcal{V}_2^2$ or to
$\mathcal{V}_1^2\oplus \mathcal{V}_2$ as a left
$\mathcal{R}$-module.
\end{lemma}
\begin{proof}
Let $Q$ be a bimodule such that $[Q]=[P]^{-1}$ in ${\rm Pic}(R)$.
Since $P$ is a finitely generated projective left module over the
finite-dimensional algebra $\mathcal{R}$, it is isomorphic to a
finite direct sum of principal indecomposable left
$\mathcal{R}$-modules, say $P\simeq \mathcal{V}_1^a\oplus
\mathcal{V}_2^b$ for some non-negative integers $a,b$. But $P$ is a
generator as a left $\mathcal{R}$-module, so $\mathcal{R}$ is a
direct summand in the left $\mathcal{R}$-module $P^m$ for some
positive integer $m$. Thus by the Krull-Schmidt Theorem, both $a$
and $b$ are positive. Similarly, $Q\simeq \mathcal{U}_1^c\oplus
\mathcal{U}_2^d$ as right $\mathcal{R}$-modules for some integers
$c,d>0$.  Now there are linear isomorphisms
$$\mathcal{R}\simeq Q\ot_{\mathcal{R}}P\simeq
(\mathcal{U}_1\ot_{\mathcal{R}}\mathcal{V}_1)^{ca}\oplus
(\mathcal{U}_1\ot_{\mathcal{R}}\mathcal{V}_2)^{cb}\oplus
(\mathcal{U}_2\ot_{\mathcal{R}}\mathcal{V}_1)^{da}\oplus
(\mathcal{U}_2\ot_{\mathcal{R}}\mathcal{V}_2)^{db}$$ Counting
dimensions and using Corollary \ref{propdimensiuni}, we see that
$(c+d)(a+b)=9$. As $a,b,c,d>0$, we must have $c+d=a+b=3$, so then
either $a=1$ and $b=2$, or $a=2$ and $b=1$.
\end{proof}

\begin{theorem}
Any invertible $\mathcal{R}$-bimodule is isomorphic either to
$_1\mathcal{R}_{\alpha}$ or to $_1{\mathcal{R}^*}_{\alpha}$ for some
$\alpha\in {\rm Aut}(\mathcal{R})$. As a consequence, ${\rm
Pic}(\mathcal{R})\simeq {\rm Out}(\mathcal{R})\times C_2$, where
$C_2$ is the cyclic group of order 2.
\end{theorem}
\begin{proof}
We know that a bimodule of type ${_1\mathcal{R}_{\alpha}}$, with
$\alpha\in {\rm Aut}(\mathcal{R})$, is invertible; the inverse of
$[{_1\mathcal{R}_{\alpha}}]$ in ${\rm Pic}(\mathcal{R})$ is
$[{_1\mathcal{R}_{\alpha^{-1}}}]$. Moreover,
$[{_1\mathcal{R}_{\alpha}}]\cdot
[{_1\mathcal{R}_{\beta}}]=[{_1\mathcal{R}_{\alpha\beta}}]$, and
$[{_1\mathcal{R}_{\alpha}}]$ depends only on the class of $\alpha$
modulo ${\rm Inn}(\mathcal{R})$.

By Corollary \ref{dualinvertible}, $\mathcal{R}^*$ is an invertible
$\mathcal{R}$-bimodule, and then so is $\mathcal{R}^*\ot_\mathcal{R}
{_1\mathcal{R}_{\alpha}} \simeq {_1\mathcal{R}^*_{\alpha}}$. Since
$\mathcal{R}^*\ot_\mathcal{R}{_1\mathcal{R}_\alpha}\simeq
{_1\mathcal{R}_\alpha}\ot_\mathcal{R}\mathcal{R}^*$ by Proposition
\ref{dualulcomuta}, and $\mathcal{R}^*\ot
_\mathcal{R}\mathcal{R}^*\simeq \mathcal{R}$ by Corollary
\ref{orderdualNakayama}, we get that the subset $\mathcal{P}$ of
${\rm Pic}(\mathcal{R})$ consisting of all
${_1\mathcal{R}_{\alpha}}$ and ${_1\mathcal{R}^*_{\alpha}}$, with
$\alpha\in {\rm Aut}(\mathcal{R})$, is a subgroup isomorphic to
${\rm Out}(\mathcal{R})\times C_2$; an isomorphism between
$\mathcal{P}$ and ${\rm Out}(\mathcal{R})\times C_2$ takes
${_1\mathcal{R}_{\alpha}}$ to $(\hat{\alpha},e)$, and
${_1\mathcal{R}^*_{\alpha}}$ to $(\hat{\alpha},c)$, where $C_2=<c>$,
$e$ is the neutral element of $C_2$, and $\hat{\alpha}$ is the class
of $\alpha$ in ${\rm Out}(\mathcal{R})$.

Let $P$ be an invertible $\mathcal{R}$-bimodule. By Lemma
\ref{invertibleleftmodule} we see that as a left
$\mathcal{R}$-module, $P$ is isomorphic either to $ \mathcal{R}$ or
to $\mathcal{R}^*$. Now Proposition \ref{isoinvertible} shows that
either $P\simeq {_1\mathcal{R}_{\alpha}}$ or $P\simeq
{_1\mathcal{R}^*_{\alpha}}$ as $\mathcal{R}$-bimodules for some
$\alpha\in {\rm Aut}(\mathcal{R})$. We conclude that ${\rm
Pic}(\mathcal{R})=\mathcal{P}$, which ends the proof.
\end{proof}

\section{Automorphisms of $\mathcal{R}$}\label{sectionautomorfisme}

The aim of this section is to compute the automorphism group  and
the group of outer automorphisms of $\mathcal{R}$. We will use a
presentation of $\mathcal{R}$ given in \cite[Remark 4.1]{dnn3},
where it is explained that $\mathcal{R}$ is isomorphic to the Morita
ring
$\left[ \begin{array}{cc} K & X\\
Y & M_2(K) \end{array} \right]$ associated with the Morita context
connecting the rings $K$ and $M_2(K)$, by the bimodules $X=K^2$ and
 $Y=M_{2,1}(K)$, with all actions given by the usual matrix
multiplication, such that both Morita maps are zero. The
multiplication of this Morita ring is given by
$$\left[ \begin{array}{cc} \al & x\\
y & f \end{array} \right]\left[ \begin{array}{cc} \al' & x'\\
y' & f' \end{array} \right]=\left[ \begin{array}{cc} \al \al' &\al x'+xf'\\
\al' y+fy' & ff' \end{array} \right]$$ for any $\al,\al'\in K$,
$f,f'\in M_2(K)$, $x,x'\in X$ and $y,y'\in Y$. This Morita ring and
$M_3(K)$ coincide as $K$-vector spaces, but they have different
multiplications. An algebra isomorphism between $\mathcal{R}$ and $\left[ \begin{array}{cc} K & X\\
Y & M_2(K) \end{array} \right]$ takes $E, (X_i)_{1\leq i\leq 2},
(Y_i)_{1\leq i\leq 2}, (F_{ij})_{1\leq i,j\leq 2}$ to the elements
in the Morita ring corresponding to the "matrix units" in $K, X, Y,
M_2(K)$. Throughout this section, we will identify $\mathcal{R}$
with $\left[ \begin{array}{cc} K & X\\
Y & M_2(K) \end{array} \right]$.

The multiplicative group $K^{\times}\times GL_2(K)$ acts on the
additive group $K^2\times M_{2,1}(K)$ by
$$(\lambda,P)\cdot (x_1,y_1)=(\lambda x_1P^{-1},Py_1)$$
for any $\lambda \in K^{\times}, P\in GL_2(K),x_1\in K^2, y_1\in
M_{2,1}(K)$, so we can form a semidirect product $(K^2\times
M_{2,1}(K))\rtimes (K^{\times}\times GL_2(K))$.

For any $x_1\in K^2,y_1\in M_{2,1}(K), \lambda \in K^{\times},P\in
GL_2(K)$ define $\varphi_{x_1,y_1,\lambda,P}:\mathcal{R}\ra
\mathcal{R}$ by
$$\varphi_{x_1,y_1,\lambda,P}(\left[ \begin{array}{cc} \al & x\\
y & f \end{array} \right])= \left[ \begin{array}{cc} \al & \al x_1+\lambda xP^{-1}-x_1PfP^{-1}\\
\al y_1+Py-PfP^{-1}y_1 & PfP^{-1} \end{array} \right].$$

\begin{theorem}
$\varphi_{x_1,y_1,\lambda,P}$ is an algebra automorphism of
$\mathcal{R}$ for any $x_1\in K^2,y_1\in M_{2,1}(K), \lambda \in
K^{\times},P\in GL_2(K)$, and $\Phi:(K^2\times M_{2,1}(K))\rtimes
(K^{\times}\times GL_2(K))\ra {\rm Aut}(\mathcal{R})$,
$\Phi(x_1,y_1,\lambda,P)=\varphi_{x_1,y_1,\lambda,P}$ is an
isomorphism of groups. An automorphism $\varphi_{x_1,y_1,\lambda,P}$
of $\mathcal{R}$ is inner if and only if $\lambda=1$. As a
consequence, ${\rm Out}(\mathcal{R})\simeq K^{\times}$.
\end{theorem}
\begin{proof}
Let $\varphi\in {\rm Aut}(\mathcal{R})$. Since the Jacobson radical
of $\mathcal{R}$ is $J(\mathcal{R})=\left[ \begin{array}{cc} 0 & X\\
Y & 0 \end{array} \right]$, $\varphi$ induces an automorphism
$\tilde{\varphi}$ of the algebra $\mathcal{R}/J(\mathcal{R})\simeq
K\times M_2(K)$, thus $\tilde{\varphi}$ acts as identity on the
first position, and as an inner automorphism associated to some
$P\in GL_2(K)$ on the second one. Lifting to $\mathcal{R}$, we see
that
$\varphi(\left[ \begin{array}{cc} 1 & 0\\
0 & 0 \end{array} \right])= \left[ \begin{array}{cc} 1 & x_1\\
y_1 & 0\end{array} \right]$ for some $x_1\in K^2$ and $y_1\in
M_{2,1}(K)$, and $\varphi(\left[ \begin{array}{cc} 0 & 0\\
0 & f \end{array} \right])= \left[ \begin{array}{cc} 0 & \mu (f)\\
\omega (f) & PfP^{-1}\end{array} \right]$ for some linear maps
$\mu:M_2(K)\ra X$ and $\omega:M_2(K)\ra Y$.

On the other hand, since $\varphi (J(\mathcal{R}))\subset
J(\mathcal{R})$, $\varphi(\left[ \begin{array}{cc} 0 & x\\
0 & 0 \end{array} \right])\in \left[ \begin{array}{cc} 0 & X\\
Y & 0 \end{array} \right]$, so then
$$\varphi(\left[ \begin{array}{cc} 0 & x\\
0 & 0 \end{array} \right])=\varphi(\left[ \begin{array}{cc} 1 & 0\\
0 & 0 \end{array} \right]\left[ \begin{array}{cc} 0 & x\\
0 & 0 \end{array} \right])\in \left[ \begin{array}{cc} 1 & x_1\\
y_1 & 0\end{array} \right]\left[ \begin{array}{cc} 0 & X\\
Y & 0 \end{array} \right]\subset \left[ \begin{array}{cc} 0 & X\\
0 & 0 \end{array} \right].$$ This shows that $\varphi(\left[ \begin{array}{cc} 0 & x\\
0 & 0 \end{array} \right])=\left[ \begin{array}{cc} 0 & \theta(x)\\
0 & 0 \end{array} \right]$ for a linear map $\theta:X\ra X$; thus
$\theta (x)=xA$ for any $x\in X$, where $A\in M_2(K)$.

Similarly we see that $\varphi(\left[ \begin{array}{cc} 0 & 0\\
y & 0 \end{array} \right])=\left[ \begin{array}{cc} 0 & 0\\
By & 0 \end{array} \right]$ for any $y\in Y$, where $B\in M_2(K)$.
Thus we obtain that $\varphi$ must be of the form
\begin{equation}\label{eqformulaphi}
\varphi(\left[ \begin{array}{cc} \al & x\\
y & f \end{array} \right])= \left[ \begin{array}{cc} \al & \al x_1+xA+\mu (f)\\
\al y_1+By+\omega(f) & PfP^{-1} \end{array} \right].
\end{equation}

By equating the corresponding entries, we see that the matrices $\varphi(\left[ \begin{array}{cc} \al & x\\
y & f \end{array} \right]\left[ \begin{array}{cc} \al' & x'\\
y' & f' \end{array} \right])$ and $\varphi(\left[ \begin{array}{cc} \al \al' &\al x'+xf'\\
y\al'+fy' & ff' \end{array} \right])$ are equal if and only if the
equations

\begin{equation}\label{eq1}
\al \mu(f')+\al
x_1Pf'P^{-1}+xAPf'P^{-1}+\mu(f)Pf'P^{-1}=xf'A+\mu(ff')
\end{equation}
and

\begin{equation} \label{eq2}
\al' \omega(f)+\al'PfP^{-1}y_1+PfP^{-1}By'+PfP^{-1}\omega (f')=Bfy'
+\omega (ff')
\end{equation}
are satisfied for any $\al,\al'\in K$, $x,x'\in K^2$, $y,y'\in
M_{2,1}(K)$, $f,f'\in M_2(K)$. If in equation (\ref{eq1}) we take
$f=0$, we get $\al (\mu(f')+x_1Pf'P^{-1})+xAPf'P^{-1}-xf'A=0$. As
this holds for any $\al\in K$, we must have
\begin{equation}\label{eq3}
\mu(f')+x_1Pf'P^{-1}=0
\end{equation}
and $x(APf'P^{-1}-f'A)=0$. As $x$ runs through $K^2$, we get
$APf'P^{-1}-f'A=0$, showing that $APf'=f'AP$ for any $f'$, so $AP\in
KI_2$, or equivalently,
\begin{equation}\label{eq4}
A\in KP^{-1}.
\end{equation}
On the other hand, it is clear that if equations (\ref{eq3}) and
(\ref{eq4}) holds, then (\ref{eq1}) is satisfied.

In a similar way, we see that (\ref{eq2}) is true if and only if
\begin{equation}\label{eq5}
\omega(f)=-PfP^{-1}y_1
\end{equation}
and
\begin{equation}\label{eq6}
B\in KP.
\end{equation}
These show that a map $\varphi$ of the form given in
(\ref{eqformulaphi}) is a ring morphism if and only if
$$\mu (f)=-x_1PfP^{-1},\; \omega (f)=-PfP^{-1}y_1,\; A\in KP^{-1},\;
B\in KP.$$ Thus take $A=\lambda P^{-1}$ and $B=\rho P$, with
$\lambda,\rho\in K$; in fact, in order for $\varphi$ to be injective
one needs $\lambda,\rho\in K^{\times}$. For any $x_1\in K^2, y_1\in
M_{2,1}(K),\lambda,\rho\in K^{\times},P\in GL_2(K)$, denote by
$\psi_{x_1,y_1, \lambda,\rho,P}:\mathcal{R}\ra \mathcal{R}$ the map
defined by
$$\psi_{x_1,y_1,
\lambda,\rho,P}(\left[ \begin{array}{cc} \al & x\\
y & f \end{array} \right])= \left[ \begin{array}{cc} \al & \al x_1+\lambda xP^{-1}-x_1PfP^{-1}\\
\al y_1+\rho Py-PfP^{-1}y_1 & PfP^{-1} \end{array} \right].$$ The
considerations above show that $\psi_{x_1,y_1, \lambda,\rho,P}$ is
an algebra endomorphism of $\mathcal{R}$. As it is clearly
injective, it is in fact an automorphism of $\mathcal{R}$. We showed
that any automorphism of $\mathcal{R}$ is one such $\psi_{x_1,y_1,
\lambda,\rho,P}$.

A straightforward computation shows that
\begin{equation}\label{eqsemidirect}
\psi_{x'_1,y'_1, \lambda',\rho',P'}\psi_{x_1,y_1,
\lambda,\rho,P}=\psi_{x'_1+\lambda'x_1(P')^{-1},y'_1+\rho'P'y_1,
\lambda'\lambda,\rho'\rho,P'P}
\end{equation}

Consider the additive group $A=K^2\times M_{2,1}(K)$ and the
multiplicative group $B=K^{\times}\times K^{\times}\times GL_2(K)$.
Then $B$ acts on $A$ by $(\lambda,\rho,P)\cdot (x_1,y_1)=(\lambda
x_1P^{-1},\rho Py_1)$, and (\ref{eqsemidirect}) shows that
$$\Psi:A\rtimes B\ra {\rm Aut}(\mathcal{R}),\,\,\, \Psi(x_1,y_1,
\lambda,\rho,P)=\psi_{x_1,y_1, \lambda,\rho,P}$$ is a group
morphism. We have also seen that $\Psi$ is surjective. Now
$\psi_{x_1,y_1, \lambda,\rho,P}$ is the identity morphism if and
only if
$$PfP^{-1}=f,\; \al x_1+\lambda xP^{-1}-x_1PfP^{-1}=x,\; \al y_1+\rho
Py-PfP^{-1}y_1=y$$ for any $\al\in K, x\in K^2,y\in M_{2,1}(K),f\in
M_2(K)$. If we take $\alpha=1, x=0, f=0$ in the second relation, we
get $x_1=0$. Hence $\lambda xP^{-1}=x$ for any $x$, so $P=\lambda
I_2$. Similarly, the third relation shows that $y_1=0$ and
$P=\rho^{-1}I_2$. Therefore ${\rm Ker}(\Psi)=0\times B_0$, where
$B_0=\{ (\lambda,\lambda^{-1},\lambda I_2)|\lambda \in
K^{\times}\}$. As $B_0$ acts trivially on $A$, the action of $B$
induces an action of the factor group $\frac{B}{B_0}$ on $A$, and
then ${\rm Aut}(\mathcal{R})\simeq \frac{A\rtimes B}{0\rtimes
B_0}\simeq A\rtimes \frac{B}{B_0}$. Denoting by $\overline{b}$ the
class of some $b\in B$ modulo $B_0$, we see that
$$\overline{(\lambda,\rho,P)}=\overline{(\rho^{-1},\rho,\rho
I_2)}\overline{(\lambda \rho,1,\rho^{-1}P)}=\overline{(\lambda
\rho^{-1},1,\rho^{-1}P)},$$ so there is a group isomorphism
$\Gamma:K^{\times}\times GL_2(K)\ra \frac{B}{B_0}$ taking
$(\lambda,P)$ to $\overline{(\lambda,1,P)}$. $\Gamma$ induces an
action of $K^{\times}\times GL_2(K)$ on $A$, given by
$$(\lambda,P)\cdot (x_1,y_1)=(\lambda,1,P)\cdot (x_1,y_1)=(\lambda
x_1P^{-1},Py_1).$$ We obtain a composition of group isomorphisms
$$\Phi:A\rtimes (K^{\times}\times GL_2(K))\longrightarrow A\rtimes
\frac{B}{B_0}\longrightarrow {\rm Aut}(\mathcal{R})$$ given by
$\Phi(x_1,y_1,\lambda,P)=\psi_{x_1,y_1, \lambda,1,P}$. Now we denote
$\psi_{x_1,y_1, \lambda,1,P}=\varphi_{x_1,y_1, \lambda,P}$ and the
first part of the statement is proved.

A direct computation shows that an element $\left[ \begin{array}{cc} \beta & z\\
g & m \end{array} \right]$ of $\mathcal{R}$ is invertible if and
only if $\beta\neq 0$ and $m\in GL_2(K)$, and in this case its
inverse is $\left[ \begin{array}{cc} \beta^{-1} & -\beta^{-1}zm^{-1}\\
-\beta^{-1}m^{-1}g & m^{-1} \end{array} \right]$, and the associated
inner automorphism of $\mathcal{R}$ takes $\left[ \begin{array}{cc} \al & x\\
y & f \end{array} \right]$ to $$\left[ \begin{array}{cc} \al & \al\beta^{-1}z+\beta^{-1}xm-\beta^{-1}zm^{-1}fm\\
-\al m^{-1}g+\beta m^{-1}y+m^{-1}fg & m^{-1}fm \end{array}
\right],$$ so it is just
$\psi_{\beta^{-1}z,-m^{-1}g,\beta^{-1},\beta,m^{-1}}$. Hence
$\varphi_{x_1,y_1, \lambda,P}=\psi_{x_1,y_1, \lambda,1,P}$ is inner
if and only if $\psi_{x_1,y_1,
\lambda,1,P}=\psi_{\beta^{-1}z,-m^{-1}g,\beta^{-1},\beta,m^{-1}}$
for some $\beta\in K^{\times}, z\in K^2,g\in M_{2,1}(K),m\in
GL_2(K)$, and taking into account the description of the kernel of
$\Psi$, this equality is equivalent to
$(x_1,y_1,\lambda,1,P)=(\beta^{-1}z,-m^{-1}g,\beta^{-1},\beta,m^{-1})(0,0,\rho,\rho^{-1},\rho
I_2)=(\beta^{-1}z,-m^{-1}g,\beta^{-1}\rho,\beta\rho^{-1},\rho
m^{-1})$ for some $\rho\in K^{\times}$. Equating the corresponding
positions, we get $1=\beta\rho^{-1}$, so $\rho=\beta$, and then
$\lambda=\beta^{-1}\rho=1$, $z=\beta x_1=\rho x_1$, $m=\rho P^{-1}$
and $g=-my_1=-\rho P^{-1}y_1$. We conclude that $\varphi_{x_1,y_1,
\lambda,P}$ is inner if and only if $\lambda=1$, and in this case,
by making the choice $\rho =1$, $\varphi_{x_1,y_1, 1,P}$ is the
inner automorphism associated
with the invertible element $\left[ \begin{array}{cc} 1 &  x_1\\
-P^{-1}y_1 & P^{-1} \end{array} \right]$.

We got that ${\rm Inn}(\mathcal{R})=\Phi (A\rtimes (1\times
GL_2(K))$, so then
$${\rm Out}(\mathcal{R})=\frac{{\rm Aut}(\mathcal{R})}{{\rm
Inn}(\mathcal{R})}\simeq \frac{A\rtimes (K^{\times}\times
GL_2(K))}{A\rtimes (1\times GL_2(K))}\simeq K^{\times}.$$ Finally,
we note that the outer automorphism corresponding to $\lambda \in
K^{\times}$ through the isomorphism ${\rm Out}(\mathcal{R})\simeq
K^{\times}$ is (the class of) $\varphi_{0,0, \lambda,I_2}$.
\end{proof}

\section{Semitrivial extensions }
\label{semitrivialextensions}

 Let $R$ be a
finite-dimensional $K$-algebra, and consider the $R$-bimodule $R^*$
with actions denoted by $\rhu$ and $\lhu$. Let $\psi:R^*\ot_RR^*\ra
R$ be a morphism of $R$-bimodules, and denote $\psi(r^*\ot _Rs^*)$
by $[r^*,s^*]$ for any $r^*,s^*\in R^*$. We say that $\psi$ is
associative if $[r^*,s^*]\rhu t^*=r^*\lhu [s^*,t^*]$ for any
$r^*,s^*,t^*\in R^*$; in other words, we have a Morita context
$(R,R,R^*,R^*,\psi,\psi)$ connecting the rings $R$ and $R$, with
both bimodules being $R^*$, and both Morita maps equal to $\psi$. It
follows from Morita theory that if $\psi$ is associative and
surjective, then it is an isomorphism of $R$-bimodules.

If $\psi:R^*\ot_RR^*\ra R$ is an associative morphism of
$R$-bimodules, we consider the semitrivial extension $R\rtimes_\psi
R^*$, which is the cartesian product $R\times R^*$ with the usual
addition, and multiplication defined by
$$(r,r^*)(s,s^*)=(rs+[r^*,s^*],(r\rhu s^*)+(r^*\lhu s))$$ for any $r,s\in
R, r^*,s^*\in R^*$. Then $R\rtimes_\psi R^*$ is an algebra with
identity element $(1,0)$; this construction was introduced in
\cite{p}. Moreover, it is a $C_2$-graded algebra, where $C_2=<c>$ is
a cyclic group of order 2, with homogeneous components
$(R\rtimes_\psi R^*)_e=R\times 0$ and $(R\rtimes_\psi R^*)_c=0\times
R^*$; here $e$ denotes the neutral element of $C_2$. It is a
strongly graded algebra if and only if $\psi$ is surjective, thus an
isomorphism.

\begin{proposition}
Let $R$ be a finite-dimensional algebra and let $\psi:R^*\ot_RR^*\ra
R$ be an associative morphism of $R$-bimodules. Then $R\rtimes_\psi
R^*$ is a symmetric algebra.
\end{proposition}
\begin{proof}
If we evaluate both sides of $[r^*,s^*]\rhu t^*=r^*\lhu [s^*,t^*]$
at 1, we get
\begin{equation} \label{formula2}
t^*([r^*,s^*])=r^*([s^*,t^*]) \mbox{ for any }r^*,s^*,t^*\in R^*.
\end{equation}

Denote $A=R\rtimes_\psi R^*$ and define
$$\Phi:A\ra A^*, \; (\Phi(r,r^*))(s,s^*)=r^*(s)+s^*(r) \;\;\mbox{for any
}r,s\in R, r^*,s^*\in R^*.$$ If $\Phi(r,r^*)=0$, then
$r^*(s)=(\Phi(r,r^*))(s,0)=0$ for any $s\in R$, so $r^*=0$, and
$s^*(r)=(\Phi(r,r^*))(0,s^*)=0$ for any $s^*\in R^*$, so $r=0$. This
shows that $\Phi$ is injective, thus a linear isomorphism. Moreover,
if $(x,x^*), (r,r^*), (s,s^*)\in A$, then

\bea (\Phi ((x,x^*)(r,r^*)))(s,s^*)&=& (\Phi(xr+[x^*,r^*],(x\rhu
r^*)+(x^*\lhu r)))(s,s^*)\\
&=&(x\rhu r^*)(s)+(x^*\lhu r)(s)+s^*(xr+[x^*,r^*]) \\
&=&r^*(sx)+x^*(rs)+s^*(xr)+s^*([x^*,r^*])\\
&=&r^*(sx)+x^*(rs)+s^*(xr)+r^*([s^*,x^*]) \;\;\;\; (\mbox{by } (\ref{formula2}))\\
&=&(s\rhu x^*+s^*\lhu x)(r)+r^*(sx+ [s^*,x^*])\\
&=&(\Phi(r,r^*))(sx+ [s^*,x^*],s\rhu x^*+s^*\lhu x)\\
&=&(\Phi(r,r^*))((s,s^*)(x,x^*))\\
&=&((x,x^*)\rhu \Phi(r,r^*))(s,s^*),\eea showing that $\Phi$ is a
morphism of left $A$-modules, and

\bea (\Phi (x,x^*)\lhu(r,r^*))(s,s^*)&=&(\Phi(x,x^*))((r,r^*)(s,s^*))\\
&=&(\Phi (x,x^*))(rs+[r^*,s^*],(r\rhu s^*)+(r^*\lhu s))\\
&=&x^*(rs)+x^*([r^*,s^*])+s^*(xr)+r^*(sx)\\
&=&x^*(rs)+s^*([x^*,r^*])+s^*(xr)+r^*(sx) \;\;\;\; (\mbox{by } (\ref{formula2}))\\
&=&(\Phi ((x,x^*)(r,r^*)))(s,s^*)\;\;\;\; (\mbox{by the computations
above}),\eea so $\Phi$ is also a morphism of right $A$-modules. We
conclude that $\Phi$ is an isomorphism of $A$-bimodules.
\end{proof}

We first mention two particular cases of interest.

The first one is for an arbitrary finite-dimensional algebra $R$ and
the zero morphism $\psi:R^*\ot_RR^*\ra R$. The associated
semitrivial extension, called in fact the trivial extension, is
$R\times R^*$ with the multiplication given by
$(r,r^*)(s,s^*)=(rs,(r\rhu s^*)+(r^*\lhu s))$ for any $r,s\in R,
r^*,s^*\in R^*$. This is just the example of Tachikawa of a
symmetric algebra constructed from $R$, see \cite[Example
16.60]{lam}.

The second one is for a symmetric finite-dimensional algebra $R$. As
$R^*\simeq R$ as bimodules, a semitrivial extension $R\rtimes_\psi
R^*$ is isomorphic to $R\times R$ with multiplication
$(r,a)(s,b)=(rs+\gamma(a\ot_Rb),rb+as)$, where $\gamma:R\ot_RR\ra R$
is a morphism of $R$-bimodules. As such a $\gamma$ is of the form
$\gamma (a\ot_Rb)=zab$ for any $a,b\in R$, where $z$ is an element
in the centre of $R$, any semitrivial extension of this kind is
isomorphic to the algebra $A_z=R\times R$ for some $z\in Cen(R)$,
whose multiplication is given by $(r,a)(s,b)=(rs+zab,rb+as)$.\\

\section{Order 2 elements in Picard groups and associative
isomorphisms}\label{sectionassociative}

In order to construct semitrivial extensions that are strongly
$C_2$-graded, we consider finite-dimensional algebras $R$ such that
$R^*\ot_RR^*\simeq R$, and we are interested in the associativity of
isomorphisms $R^*\ot_RR^*\ra R$. We have seen in Proposition
\ref{dualinvertible} that such an $R$ is necessarily
quasi-Frobenius, it is clear that $[R^*]$ has order at most 2 in
${\rm Pic}(R)$, and we addressed Question 2 in the Introduction,
asking whether any such isomorphism is associative.

The following shows that the answer to the question depends only on
the algebra, and not on a particular choice of the isomorphism.

\begin{proposition}\label{anyisoassoc}
If $R$ is a finite-dimensional algebra such that $R^*\ot_RR^*\simeq
R$ as bimodules and there exists an associative isomorphism
$R^*\ot_RR^*\ra R$, then any other such isomorphism  is associative.
\end{proposition}
\begin{proof}
Let $\psi,\psi':R^*\ot_RR^*\ra R$  be isomorphisms of bimodules, and
assume that $\psi$ is associative. Then $\psi'\psi^{-1}$ is an
automorphism of the bimodule $R$, so it is the multiplication by a
central invertible element $c$. Therefore $\psi'(y)=c\psi(y)$ for
any $y\in R^*\ot_RR^*$. We note that $c\rhu r^*=r^*\lhu c$ for any
$r^*\in R^*$, since $(c\rhu r^*)(a)=r^*(ac)=r^*(ca)=(r^*\lhu c)(a)$
for any $a\in R$.

Now for any $r^*,s^*,t^*\in R^*$ \bea \psi'(r^*\ot_Rs^*)\rhu
t^*&=&(c\psi(r^*\ot_Rs^*))\rhu t^*\\
&=&c\rhu (\psi(r^*\ot_Rs^*)\rhu t^*)\\
&=&c\rhu (r^*\lhu \psi(s^*\ot_Rt^*)) \\
&=&(c\rhu r^*)\lhu \psi(s^*\ot_Rt^*)\\
&=&( r^*\lhu c)\lhu \psi(s^*\ot_Rt^*)\\
&=&r^*\lhu (c\psi(s^*\ot_Rt^*))\\
&=&r^*\lhu  \psi'(s^*\ot_Rt^*), \eea showing that $\psi'$ is
associative as well.
\end{proof}

The following answers in the positive our question in the Frobenius
case.

\begin{lemma} \label{Frobeniusasociativ}
Let $R$ be a Frobenius algebra such that $R^*\ot_RR^*\simeq R$ as
bimodules. Then any isomorphism $\psi:R^*\ot_RR^*\ra R$ is
associative.
\end{lemma}
\begin{proof}
Let $\lambda\in R^*$ be a Frobenius form and let $\nu$ be the
Nakayama automorphism associated with $\lambda$. We have seen in
Proposition \ref{dualalpha} that $\theta:{_1R_\nu}\ra R^*$, $\theta
(r)=r\rhu \lambda$, is a bimodule isomorphism. Then
$R^*\ot_RR^*\simeq {_1R_\nu}\ot_R{_1R_\nu}\simeq {_1R_{\nu^2}}$, so
 ${_1R_{\nu^2}}\simeq R$, which shows that $\nu^2$ is inner; let
$\nu^2 (r)=u^{-1}ru$ for any $r\in R$, where $u$ is an invertible
element of $R$. Now for any $a\in R$

\bea \lambda (au)&=&\lambda (u\nu(a)) \;\;\;\;\;\;\;\;\;\;\;\;\;\;
(\mbox{since }\nu\mbox{
is the Nakayama automorphism})\\
&=&\lambda(u\nu^2(\nu^{-1}(a)))\\
&=&\lambda(uu^{-1}\nu^{-1}(a)u) \;\;\;(\mbox{since }\nu^2\mbox{ is
inner})\\
&=&\lambda(\nu^{-1}(a)u)\\
&=&\lambda(ua) \;\;\;\;\;\;\;\;\;\;\;\;\;\;\;\;\;\;\;\; (\mbox{since
}\nu\mbox{ is the Nakayama automorphism}), \eea showing that
$\lambda (au)=\lambda(ua)$, or equivalently, $u\rhu
\lambda=\lambda\lhu u$. Therefore $\theta(u)=u\rhu
\lambda=\lambda\lhu u=\nu(u)\rhu \lambda=\theta(\nu(u))$, so
$\nu(u)=u$, since $\theta$ is injective.

It is easy to check that $\delta:{_1R_\nu}\ot_R{_1R_\nu}\ra
{_1R_{\nu^2}}$, $\delta (r\ot_Rs)=r\nu (s)$ for any $r,s\in R$, and
$\omega:{_1R_{\nu^2}}\ra R$, $\omega(r)=ru^{-1}$ for any $r\in R$,
are both bimodule isomorphisms. Composing them, we obtain an
isomorphism $F=\omega\delta:{_1R_\nu}\ot_R{_1R_\nu}\ra R$,
$F(r\ot_Rs)=r\nu (s)u^{-1}$. Denoting by $*$ the right action of $R$
on $_1R_\nu$, we see that for any $r,s,t\in R$ \bea
r* F(s\ot_Rt)&=&r* (s\nu (t)u^{-1})\\
&=&r\nu(s\nu (t)u^{-1})\\
&=&r\nu(s)u^{-1}tu\nu(u^{-1})\;\;\;\;\;\; \mbox{(since }
\nu^2(t)=u^{-1}tu\mbox{)} \\
&=&r\nu(s)u^{-1}t\;\;\;\;\;\;\;\;\;\;\;\;\;\;\;\;\;\;\; \mbox{(since }\nu(u)=u\mbox{)}\\
&=&F(r\ot_Rs)t \eea

Since $_1R_\nu\simeq R^*$, $F$ induces an associative bimodule
isomorphism $F':R^*\ot_RR^*\ra R$, and then any such isomorphism is
associative by Proposition \ref{anyisoassoc}.
\end{proof}

In the initial version of the paper, we were able to answer Question
2 only in the Frobenius case. One of the referees indicated us how
the quasi-Frobenius case can be derived from the Frobenius one. Some
of the steps of the approach were explained in Section
\ref{sectionpicardquasi}, and we show below how the conclusion can
be reached. We keep the notation at the beginning of Section
\ref{sectionpicardquasi}, where $R$ and $S$ are two Morita
equivalent algebras, $(R,S,P,Q,[,],(,))$ is a strict Morita context
connecting them, $F=Q\ot_R(-)\ot_RP$ is the induced monoidal
equivalence between their categories of bimodules, and $G$ is its
quasi-inverse. Denote by $\theta_M:F(M)\ot_SF(M)\ra F(M\ot_R M)$,
$$\theta_M(q_1\ot_R
m_1\ot_Rp_1\ot_Sq_2\ot_Rm_2\ot_Rp_2)=q_1\ot_Rm_1[p_1,q_2]\ot_Rm_2\ot_Rp_2,$$
and $\mu:F(R)\ra S$, $\mu(q\ot_Rr\ot_Rp)=(qr,p)$, the isomorphisms
of $S$-bimodules associated with this monoidal equivalence.

If $M$ is an $R$-bimodule and $\psi:M\ot_RM\ra R$ is an $R$-bimodule
isomorphism, consider the isomorphism $\tilde{\psi}=\mu
F(\psi)\theta_M:F(M)\ot_SF(M)\ra S$.

\begin{proposition} \label{proptransferasociativ}
{\rm (i)} The mapping $\psi\mapsto \tilde{\psi}$ is a bijective
correspondence between the isomorphisms of $R$-bimodules
$M\ot_RM\stackrel{\sim}{\ra} R$ and the isomorphisms of
$S$-bimodules $F(M)\ot_SF(M)\stackrel{\sim}{\ra} S$.\\
{\rm (ii)} If $\psi:M\ot_RM\ra R$ is an associative isomorphism of
$R$-bimodules, then $\tilde{\psi}$ is associative.
\end{proposition}
\begin{proof}
(i) follows immediately since $F$ is full and faithful. For (ii),
let  $z_i=q_i\ot_Rm_i\ot_Rp_i\in F(M)$ for $1\leq i\leq 3$. Then
$\tilde{\psi}(z_1\ot_Sz_2)=(q_1\psi(m_1[p_1,q_2]\ot_Rm_2),p_2)$, so
\bea
\tilde{\psi}(z_1\ot_Sz_2)z_3&=&(q_1\psi(m_1[p_1,q_2]\ot_Rm_2),p_2)q_3\ot_Rm_3\ot_Rp_3\\
&=&q_1\psi(m_1[p_1,q_2]\ot_Rm_2)[p_2,q_3]\ot_Rm_3\ot_Rp_3\\
&=&q_1\ot_R\psi(m_1\ot_R[p_1,q_2]m_2)[p_2,q_3]m_3\ot_Rp_3 \\
&=&q_1\ot_Rm_1\psi([p_1,q_2]m_2[p_2,q_3]\ot_Rm_3)\ot_Rp_3 \\
&=&q_1\ot_Rm_1\ot_R [p_1,q_2]\psi(m_2[p_2,q_3]\ot_Rm_3)p_3 \\
&=&q_1\ot_Rm_1\ot_R p_1(q_2\psi(m_2[p_2,q_3]\ot_Rm_3),p_3) \\
&=&z_1\tilde{\psi}(z_2\ot_S z_3)\eea An alternative proof can be
done with a categorical approach,  using the fact that $F$ is a
monoidal equivalence and showing the commutativity of some diagrams.
\end{proof}

\begin{corollary} \label{corolartransferasociativ}
Let $M$ be an $R$-bimodule. Then any isomorphism of $R$-bimodules
$M\ot_RM\stackrel{\sim}{\ra} R$ is associative if and only if any
isomorphism of $S$-bimodules $F(M)\ot_SF(M)\stackrel{\sim}{\ra} S$
is associative. \end{corollary}
\begin{proof} The only if
part follows directly from Proposition \ref{proptransferasociativ},
while the if part follows by applying the direct implication for the
quasi-inverse $G$ of $F$ and the $S$-bimodule $F(M)$.
\end{proof}

Now we can answer Question 2 in the general.

\begin{proposition} \label{quasiFrobeniusasociativ}
Let $R$ be a finite-dimensional algebra such that $R^*\ot_RR^*\simeq
R$ as bimodules. Then any isomorphism $\psi:R^*\ot_RR^*\ra R$ is
associative.
\end{proposition}
\begin{proof}
We have seen that $R$ is necessarily quasi-Frobenius. Let $S$ be a
basic algebra of $R$. Then $S$ is Frobenius. We know that
$F(R^*)\simeq S^*$ by Proposition \ref{3.2}, and that
$S^*\ot_SS^*\simeq S$ by Proposition \ref{proptransferasociativ}(i)
(or alternatively, by Corollary \ref{3.4}). By Lemma
\ref{Frobeniusasociativ}, any isomorphism
$S^*\ot_SS^*\stackrel{\sim}{\ra} S$ is associative. Corollary
\ref{corolartransferasociativ} shows now that any isomorphism
$R^*\ot_RR^*\stackrel{\sim}{\ra} R$ is associative.
\end{proof}

As a consequence, we can construct semitrivial extensions that are
strongly graded algebras from any algebra whose dual has order at
most 2 in the Picard group.

\begin{proposition} \label{propositionconstruction}
Let $R$ be a finite-dimensional algebra such that there exists an
isomorphism of $R$-bimodules $\psi:R^*\ot_RR^*\ra R$. Then
$A=R\rtimes_\psi R^*$ is a symmetric algebra and a strongly
$C_2$-graded algebra with grading given by $A_e=R\rtimes 0$ and
$A_c=0\rtimes R^*$.
\end{proposition}

In the particular case where $R=\mathcal{R}$ is Nakayama's
9-dimensional algebra, and
$\varphi:\mathcal{R}^*\ot_\mathcal{R}\mathcal{R}^*\ra \mathcal{R}$
is an isomorphism of bimodules, for example the one described in
Remark \ref{remarcavarphi}, we obtain an example answering in the
negative our initial Question 1, for both the symmetric property and
the Frobenius property.

\begin{corollary} \label{corolarexemplu}
$\mathcal{R}\rtimes_{\varphi} \mathcal{R}^*$  is a symmetric
strongly $C_2$-graded algebra, whose homogeneous component of
trivial degree is not Frobenius.
\end{corollary}

This example also answers a question posed by the referee of our
paper \cite{dnn2}. It was proved in \cite[Proposition 2.1]{dnn2}
that if $B$ is a subalgebra of a Frobenius algebra $A$, such that
$A$ is free as a left $B$-module and also as a right $B$-module,
then $B$ is Frobenius, too. The question was whether the conclusion
remains valid if we only suppose that $A$ is projective as a left
$B$-module and as a right $B$-module. The example constructed in
Corollary \ref{corolarexemplu} shows that the answer is negative.
Indeed, $A$ is even symmetric, and it is projective as a left
$A_e$-module and as a right $A_e$-module, however $A_e$ is not a
Frobenius algebra.

As another consequence of Proposition \ref{propositionconstruction},
we obtain a class of examples of strongly graded algebras that are
symmetric as algebras, while their homogeneous component of trivial
degree is Frobenius, but not symmetric. Indeed, we can take
 a Frobenius algebra $R$ such that the order of
 $[R^*]$ in ${\rm Pic}(R)$ is $2$; in other words, the Nakayama automorphism $\nu$
 with respect to a Frobenius form is not inner, but $\nu^2$ is
 inner. Then there is an isomorphism of
 $R$-bimodules $\psi:R^*\ot_RR^*\ra R$, and by
 Lemma \ref{Frobeniusasociativ}, it is associative. Hence we can form the semitrivial
 extension $R\rtimes_{\psi}R^*$,  a strongly $C_2$-graded
 algebra which is symmetric, and its homogeneous component of
 trivial degree is isomorphic to $R$, which is Frobenius, but not
 symmetric.

 We have several classes of examples of Frobenius algebras $R$ such
 that
 $[R^*]$ has order $2$ in ${\rm Pic}(R)$:

 (i) A first class follows from Example \ref{exempleordinR^*}. For $R=H_1(C,n,c,c^*)$, the order of $[R^*]$ is $2$ if and only if
$n=2$. Thus we obtain such an $R$ if we have a finite abelian group
$C$, an element $c\in C$ with $c^2\neq 1$, and a linear character
$c^*\in C^*$ such that $(c^*)^2=1$ and $c^*(c)=-1$. A particular
family of such examples is when we take $C=<c>\simeq C_{2r}$, where
$r\geq 2$, and $c^*\in C^*$ defined by $c^*(c)=-1$, obtaining a Hopf
algebra of dimension $4r$, generated by the grouplike element $c$
and the $(1,c)$-skew-primitive element $x$, subject to relations
$c^{2r}=1, x^2=c^2-1, xc=-cx$.

(ii) A second class follows from Example \ref{exempleordinR^*}, too.
For $R=H_2(C,n,c,c^*)$, the order of $[R^*]$ is $2$ if and only if
$\frac{m}{(\frac{m}{n},n-1)}=2$, where $m$ is the order of $c^*$. It
is easy to check that this happens if and only if $m=n=2$. Thus we
need a finite abelian group $C$, a character $c^*\in C^*$ such that
$(c^*)^2=1$ and an element $c\in C$ such that $c^*(c)=-1$ (in
particular, the order of $c$ must be even). A particular family of
such examples is when we take $C=<c>\simeq C_{2r}$, where $r\geq 1$,
and $c^*\in C^*$ defined by $c^*(c)=-1$, obtaining a Hopf algebra of
dimension $4r$, generated by the grouplike element $c$ and the
$(1,c)$-skew-primitive element $x$, subject to relations $c^{2r}=1,
x^2=0, xc=-cx$. For $r=1$ this is just Sweedler's $4$-dimensional
Hopf algebra.

(iii) Another example is $R_{-1}=K_{-1}[X,Y]/(X^2,Y^2)$  from
Example \ref{exemplequantumplane} for $q=-1$.

(iv) Let $H$ be a unimodular finite-dimensional Hopf algebra, i.e.,
the spaces of left integrals and right integrals coincide in $H$;
equivalently, the unimodular element $\mathcal{G}$ is trivial. By
Radford's formula, see \cite[Theorem 12.10]{lorenz} or \cite[Theorem
10.5.6]{radford2}, $S^4(h)=a^{-1}ha$ for any $h\in H$, where $a$ is
the modular element of $H^*$ regarded inside $H$ via the isomorphism
$H\simeq H^{**}$. Thus $S^4$ is inner, and then the order of $S^2$
in ${\rm Out}(H)$ is either $1$ or $2$. By Theorem \ref{ordinH*}, in
the first case $[H^*]$ has order $1$ in ${\rm Pic}(H)$, and $H$ is
symmetric, while in the second case, $[H^*]$ has order $2$ in ${\rm
Pic}(H)$. We conclude that a class of Frobenius algebras as we are
looking for is the family of all unimodular finite-dimensional Hopf
algebras that are not symmetric. A class of such objects was
explicitly
constructed in \cite{suzuki}. \\

{\bf Acknowledgement.} We thank the referees for the comments,
corrections and suggestions that improved the exposition of the
paper and strengthened some of the results, in particular for
indicating a method for answering Question 2 in the general. The
first two authors were supported by a grant of UEFISCDI, project
number PN-III-P4-PCE-2021-0282, contract PCE 47/2022.


\begin{thebibliography}{99}

\bibitem{as} N. Andruskiewitsch, H.-J. Schneider,
Lifting of quantum linear spaces and pointed Hopf algebras of order
$p^3$, J. Algebra {\bf 209} (1998), 659-691.



\bibitem{bass} H. Bass, Algebraic K-Theory, W. A. Benjamin, Inc.,
1968.

\bibitem{bdg}M. Beattie,   S. D\u{a}sc\u{a}lescu, L. Gr\"{u}nenfelder,
Constructing pointed Hopf algebras by Ore extensions, J. Algebra,
{\bf 225} (2000), 743-770.

\bibitem{bichon} J. Bichon, Cosovereign Hopf algebras, J. Pure Appl.
Algebra {\bf 157} (2001), 121-133.


\bibitem{dnn1} S. D\u{a}sc\u{a}lescu, C. N\u{a}st\u{a}sescu and L.
N\u{a}st\u{a}sescu, Frobenius algebras of corepresentations and
group graded vector spaces, J. Algebra {\bf 406} (2014), 226-250.


\bibitem{dnn2} S. D\u{a}sc\u{a}lescu, C. N\u{a}st\u{a}sescu and L.
N\u{a}st\u{a}sescu, Hopf algebra actions and transfer of Frobenius
and symmetric properties, Math. Scand. {\bf 126} (2020), 32-40.

\bibitem{dnn3} S. D\u{a}sc\u{a}lescu, C. N\u{a}st\u{a}sescu and L.
N\u{a}st\u{a}sescu, On a class of quasi-Frobenius algebras, J. Pure
Appl. Algebra {\bf 226} (2022), 106992.

\bibitem{fss} J. Fuchs, G. Schaumann and C. Schweigert, Eilenberg-Watts calculus for finite categories
and a bimodule Radford $S^4$ theorem, Trans. Amer. Math. Soc. {\bf
373} (2020), 1-40.

\bibitem{lam} T. Y. Lam, Lectures on modules and rings, GTM {\bf
189}, Springer Verlag, 1999.

\bibitem{lorenz} M. Lorenz, A tour of representation theory, AMS
Graduate Studies in Mathematics {\bf 193}, 2018.

 \bibitem {nak} T. Nakayama, On Frobeniusean algebras. I, Ann. of Math. {\bf 40} (1939), 611-633.

\bibitem{nn} T. Nakayama, C. Nesbitt, Note on symmetric algebras, Ann. of
Math. {\bf 39} (1938), 659-668.

\bibitem{nvo} C. N\u{a}st\u{a}sescu and F. van Oystaeyen, Methods
of graded rings, Lecture Notes in Math., vol. 1836 (2004), Springer
Verlag.

\bibitem{os} U. Oberst, H.-J. Schneider, ${\rm \ddot{U}}$ber
Untergruppen endlicher algebraischer Gruppen, Manuscripta Math. {\bf
8} (1973), 217-241.

\bibitem{p} I. Palm${\rm \acute{e}}$r, The global homological
dimension of semi-trivial extensions of rings, Math. Scand. {\bf 37}
(1975), 223-256.

\bibitem{radford} D. E. Radford, The trace function and Hopf
algebras, J. Algebra {\bf 163} (1994), 583-622.

\bibitem{radford2} D. E. Radford, Hopf algebras. Series on Knots and Everything, 49. World Scientific Publishing Co. Pte. Ltd., Hackensack, NJ, 2012.


\bibitem{sy} A. Skowro\'{n}ski, K. Yamagata, Frobenius algebras I.
Basic representation theory, European Mathematical Society, 2012.

\bibitem{suzuki} S. Suzuki, Unimodularity of finite dimensional Hopf
algebras, Tsukuba J. Math. {\bf 20} (1996), 231-238.

\bibitem{takeuchi} M. Takeuchi, $\sqrt{\mathrm{Morita}}$ theory - formal ring laws and monoidal equivalences of categories
of bimodules, J. Math. Soc. Japan {\bf 39} (1987), 301-336.

\end{thebibliography}
\end{document}